\documentclass[11pt]{article}
\usepackage[dvips,letterpaper,margin=1.0in]{geometry}
\usepackage{sty}

\usepackage{authblk}
\author[1,2]{Ankur Moitra\thanks{Email: \textit{moitra@mit.edu}. Partially supported by a Microsoft Trustworthy AI Grant, an ONR grant and a David and Lucile Packard Fellowship.}}
\author[3]{Alexander S.\ Wein\thanks{Email: \textit{aswein@ucdavis.edu}. Partially supported by an Alfred P.\ Sloan Research Fellowship and NSF CAREER Award CCF-2338091.}}

\affil[1]{Department of Mathematics, MIT}
\affil[2]{Computer Science and Artificial Intelligence Lab, MIT}
\affil[3]{Department of Mathematics, UC Davis}

\date{}

\title{Precise Error Rates for Computationally Efficient Testing}

\begin{document}

\maketitle

\begin{abstract}
We revisit the fundamental question of simple-versus-simple hypothesis testing with an eye towards computational complexity, as the statistically optimal likelihood ratio test is often computationally intractable in high-dimensional settings. In the classical spiked Wigner model with a general i.i.d.\ spike prior we show (conditional on a conjecture) that an existing test based on linear spectral statistics achieves the best possible tradeoff curve between type I and type II error rates among all computationally efficient tests, even though there are exponential-time tests that do better. This result is conditional on an appropriate complexity-theoretic conjecture, namely a natural strengthening of the well-established low-degree conjecture. Our result shows that the spectrum is a sufficient statistic for computationally bounded tests (but not for all tests).

To our knowledge, our approach gives the first tool for reasoning about the precise asymptotic testing error achievable with efficient computation. The main ingredients required for our hardness result are a sharp bound on the norm of the low-degree likelihood ratio along with (counterintuitively) a positive result on achievability of testing. This strategy appears to be new even in the setting of unbounded computation, in which case it gives an alternate way to analyze the fundamental statistical limits of testing.
\end{abstract}

\newpage

\tableofcontents

\newpage

\section{Introduction}

A fundamental question in mathematical statistics is that of testing between two simple hypotheses, that is, the task of deciding which of two known distributions produced a given sample. The celebrated Neyman--Pearson lemma states that tests based on thresholding the likelihood ratio are optimal in the sense that they sweep out the best possible tradeoff between type I and type II error rates as the threshold is varied. However, in high-dimensional settings, the likelihood ratio is often intractable to compute because it involves summing over an exponential number of possible values for a latent variable. In this work we revisit the classical problem of simple-versus-simple testing with an eye towards computational complexity: we aim to precisely characterize the best possible testing error \emph{achievable by a computationally efficient algorithm}, which may differ from the statistical limit.

\subsection{Spiked Wigner model} While our methods have the potential to be applied more broadly, we focus on one canonical high-dimensional testing problem: the \emph{spiked Wigner model}.

\begin{definition}[Spiked Wigner testing problem]
\label{def:wigner}
For a positive integer $n$, a signal-to-noise ratio (SNR) $\lambda \ge 0$, and a spike prior $\pi$ which is a distribution on $\RR$ with mean 0 and variance 1, let $\PP_{\lambda,\pi,n}$ denote the distribution over $n \times n$ symmetric matrices $Y$ generated as
\begin{equation}\label{eq:def-wigner}
Y = \lambda \cdot \frac{xx^\top}{\|x\|^2} + W
\end{equation}
where $x \in \RR^n$ has entries drawn i.i.d.\ from $\pi$, and $W \in \RR^{n \times n}$ is drawn (independent from $x$) according to the Gaussian Orthogonal Ensemble (GOE): $W$ is symmetric with $\{W_{ij} : i \le j\}$ independent, $W_{ij} \sim \cN(0,1/n)$ for $i \ne j$, and $W_{ii} \sim \cN(0,2/n)$. We adopt the convention $xx^\top/\|x\|^2 = 0$ in the (unlikely) event $\|x\| = 0$. For ease of notation, we will suppress the dependence on $\pi,n$ and write $\PP_\lambda = \PP_{\lambda,\pi,n}$. Our focus is on the asymptotic regime $n \to \infty$ with $\lambda,\pi$ held fixed. We consider testing the null hypothesis $Y \sim \PP_0$ (the case $\lambda = 0$) against a specific alternative $Y \sim \PP_\lambda$, where $\lambda,\pi$ are known.
\end{definition}

\noindent This testing problem amounts to detecting the presence of a rank-1 ``signal'' buried in random ``noise.'' We emphasize that only a \emph{single} sample $Y \in \RR^{n \times n}$, drawn either from $\PP_0$ or $\PP_\lambda$, is observed.

\emph{Spiked} (or \emph{deformed}) random matrix models such as~\eqref{eq:def-wigner} and the related \emph{spiked Wishart model} have been extensively studied from the perspective of random matrix theory~\cite{johnstone,BBP,BS-spiked,peche,paul,FP-wigner,maida,BY-limit,CDF-wigner,nadler,BN-eigenvec,PRS-wigner}. Notably, the model~\eqref{eq:def-wigner} undergoes a spectral transition akin to that of Baik, Ben Arous, and P\'ech\'e~\cite{BBP} at the threshold $\lambda = 1$: when $\lambda \le 1$, the empirical distribution of eigenvalues converges to the Wigner semicircle law (supported on $[-2,2]$) and the maximum eigenvalue converges to $2$; when $\lambda > 1$, the maximum eigenvalue converges to $\lambda + 1/\lambda > 2$ due to a single ``signal'' eigenvalue that exits the semicircular bulk~\cite{FP-wigner,maida}. As an immediate corollary, thresholding the maximum eigenvalue gives a test achieving \emph{strong detection} when $\lambda > 1$, that is, testing $\PP_0$ versus $\PP_\lambda$ with both type I and type II error probabilities tending to $0$ as $n \to \infty$.

More recently, spiked models have been studied from a statistical perspective with the aim of identifying the best possible test, which may not be based on the maximum eigenvalue~\cite{sphericity,sphericity-2,MRZ,LKZ-sparse,BMVVX,opt-subopt,fund-limits-wigner,weak-wigner,det-rect}. For the moment, we restrict ourselves to \emph{spectral tests}, that is, tests that only use the spectrum (multiset of eigenvalues) of $Y$. In the Wigner model (Definition~\ref{def:wigner}), no spectral test can achieve strong detection when $\lambda < 1$~\cite{MRZ} but \emph{weak detection}, i.e., testing with a non-vanishing advantage over random guessing, is possible for any $\lambda > 0$ by thresholding the trace of $Y$. More precisely, the best possible asymptotic tradeoff curve between type I and type II error rates for spectral tests is known for each $\lambda < 1$~\cite{BL-free-1,BL-free-2} (see~\cite{weak-wigner}), and is furthermore achieved by a computationally efficient (polynomial-time) test based on \emph{linear spectral statistics (LSS)}~\cite{weak-wigner}. Specifically,~\cite{weak-wigner} considers a family of tests, henceforth \emph{LSS}, based on thresholding the value
\begin{equation}\label{eq:lss}
\sum_{i=1}^n h_\lambda(\mu_i) \qquad \text{with} \qquad h_\lambda(\mu) := -\log(1 - \lambda \mu + \lambda^2),
\end{equation}
where $\mu_1 \ge \mu_2 \ge \cdots \ge \mu_n$ are the eigenvalues of $Y$, and the optimal function $h_\lambda$ has been carefully chosen. The LSS paradigm has been studied in a long line of work, including~\cite{clt-wishart,clt-band,clt-wigner,BL-free-1,banerjee-ma,weak-wigner,weak-wigner-general-rank,det-rect,lss-block,lss-div-spikes,global-local}.

However, the spectrum of $Y$ only contains information about the norm $\lambda$ of the rank-1 signal term, not its direction $x/\|x\|$. Conceivably a better test can be constructed by exploiting the eigenvectors of $Y$ in addition to the eigenvalues. Indeed, for some (but not all) spike priors $\pi$, there are (non-spectral) tests that achieve strong detection for some $\lambda < 1$~\cite{BMVVX,opt-subopt}, beating the spectral threshold. For instance, if $\pi$ places enough mass on $0$ so that $x$ is sufficiently sparse, strong detection is possible below the spectral threshold by exhaustively enumerating all possible sparsity patterns. For any prior $\pi$ with bounded support,~\cite{fund-limits-wigner} resolves the optimal statistical performance among all possible tests, identifying the exact threshold $\lambda^* = \lambda^*(\pi) \le 1$ above which strong detection is information-theoretically possible (confirming a conjecture of~\cite{LKZ-sparse}), and showing that the error tradeoff (between types I and II) of LSS is information-theoretically optimal when $\lambda < \lambda^*$.

While the above results resolve the fundamental statistical limits of the spiked Wigner testing problem, the computational complexity remains unclear: for all $\pi$ and all $\lambda < 1$, the best known \emph{computationally efficient} test is LSS. In other words, the best known computationally efficient tests use only the spectrum, suggesting that the spectrum may be a ``\emph{computationally sufficient statistic}.'' In this work we aim to confirm this by answering (affirmatively) the following question:
\begin{center}
\emph{Do linear spectral statistics (LSS) achieve the best possible tradeoff between type I and type II error rates among all polynomial-time tests for all $\pi$ and all $\lambda < 1$?}
\end{center}
The regime of interest for this question is $\lambda^* < \lambda < 1$, where strong detection is information-theoretically possible but all known algorithms achieving this require exponential time. For instance, optimal statistical performance is of course achieved by thresholding the likelihood ratio
\begin{equation}\label{eq:wig-lr}
\frac{d\PP_\lambda}{d\PP_0}(Y) = \E_{x \stackrel{\mathrm{iid}}{\sim} \pi} \exp\left(-\frac{n}{4}\left\|Y - \lambda\frac{xx^\top}{\|x\|^2}\right\|_\F^2 + \frac{n}{4}\|Y\|_\F^2\right),
\end{equation}
but naive evaluation of this quantity requires a sum over $\exp(\Omega(n))$ possible values for $x$ (or an integral of comparable complexity, if $\pi$ is not discrete).

\subsection{Average-case complexity} Our main question above lies in the realm of \emph{average-case computational complexity}, as we are interested in inherent limitations on both the runtime and statistical properties of algorithms over a particular distribution of random inputs. While such considerations are prevalent in high-dimensional statistics, we currently lack the tools to prove a lower bound on the runtime of \emph{all} algorithms in these settings, even under a standard complexity assumption such as $\mathsf{P} \ne \mathsf{NP}$. Instead, a rich landscape of frameworks have emerged for giving various forms of ``rigorous evidence'' for computational hardness of statistical tasks. Most approaches either establish hardness conditional on the assumed hardness of a ``standard'' statistical problem such as \emph{planted clique}, or establish unconditional failure of some restricted class of methods such as statistical query (SQ) algorithms or the sum-of-squares (SoS) hierarchy. Some pioneering works in this area include~\cite{jerrum,decelle,sq-clique,BR-reduction,GS-ogp,sos-clique}, and we also refer the reader to~\cite{sos-survey,ld-notes,secret-leakage,phys-survey} for further exposition.

In this work we explore inherent computational barriers from the perspective of the \emph{low-degree polynomial framework}, which first arose from the work of~\cite{sos-clique,HS-bayesian,sos-detect,hopkins-thesis} (see also the survey~\cite{ld-notes}) and was later refined and extended in various directions, e.g.~\cite{SW-estimation,GJW-ld,planted-v-planted,coloring-clique}. This approach amounts to studying the power and limitations of algorithms that can be represented as low-degree polynomial functions of the input variables (in our case, the entries of $Y$). The polynomial degree is thought of as a measure of the algorithm's complexity and a proxy for runtime, with degree $D$ corresponding to runtime roughly $n^D$ (up to $\log n$ factors in the exponent), which is the number of terms in such a polynomial. We refer to the \emph{low-degree conjecture} as the informal belief that the class of degree-$D$ polynomials is as powerful as \emph{all} algorithms of the corresponding runtime, for a particular style of high-dimensional testing problems. This heuristic is by now well-established as a reliable method for predicting and rigorously vindicating suspected computational barriers in a wide array of statistical problems. For instance, in the spiked Wigner testing problem that we consider in this work, with any prior $\pi$ of bounded support, the following phase transition is known for low-degree polynomials: if $\lambda < 1$ then \emph{any} degree-$o(n/\log n)$ polynomial \emph{fails} to achieve strong detection (in an appropriate sense), whereas degree $\omega(1)$ suffices when $\lambda > 1$~\cite{ld-notes}. We consider this as rigorous evidence that strong detection requires exponential runtime $\exp(n^{1-o(1)})$ when $\lambda < 1$, at least for a broad class of known approaches. In contrast, recall that strong detection is possible in polynomial time when $\lambda > 1$.

\subsection{Our contributions} While a myriad of prior work has shown hardness in the low-degree framework for strong (or weak) detection in various models, in this work we seek to answer an even more precise question, as we are concerned with the \emph{exact error probability} in a regime where the type I and type II error rates will both converge to nontrivial constants. To our knowledge, no existing tools allow us to approach this question --- not in the low-degree framework nor in any other framework for average-case complexity. Our main contribution is a new way to argue, based on low-degree polynomials, about the exact error probability achievable by computationally efficient algorithms. We illustrate this with the following result in the spiked Wigner model.

\begin{theorem}[Main result, informal]
\label{thm:main-informal-comp}
Consider the spiked Wigner testing problem (Definition~\ref{def:wigner}) with any prior $\pi$ of bounded support and any $\lambda < 1$. Assuming a natural strengthening of the low-degree conjecture (Conjecture~\ref{conj:new}), any test with error tradeoff (between types I and II) asymptotically better than LSS requires runtime $\exp(n^{1-o(1)})$.
\end{theorem}

\noindent The precise statement of this result is presented in Section~\ref{sec:our-approach-comp}, particularly Corollary~\ref{cor:main-comp}.

One consequence of our result is a certain \emph{computational universality} with respect to the prior. While the optimal statistical performance (namely the threshold $\lambda^*$ for strong detection) depends on $\pi$, our result shows that the best \emph{computationally efficient} test only uses the spectrum and thus its performance depends only on $\lambda$ (not $\pi$). Put another way, the spectrum is a \emph{sufficient statistic} for computationally-efficient testing. It was known previously that the low-degree threshold for strong detection does not depend on the prior~\cite{ld-notes}, and our result extends this to weak detection.

The computational universality for testing is in contrast to the related task of \emph{estimating} the rank-1 spike. For the estimation problem, the best known computationally efficient algorithm uses \emph{approximate message passing (AMP)}~\cite{amp,BM-amp}, which achieves nontrivial mean squared error whenever $\lambda > 1$ and this error depends on both $\lambda$ and $\pi$~\cite{FR-amp,MV-amp}. For more on the estimation problem, we refer the reader to~\cite{miolane-survey} and references therein.

The argument we use to prove Theorem~\ref{thm:main-informal-comp} can be specialized to the case where there is no restriction on runtime, in which case it yields a new approach for establishing the statistically optimal error tradeoff. To our knowledge, this method has not appeared before in the literature. We illustrate this by giving an alternate re-proof of the following result of~\cite{ALR,fund-limits-wigner}.

\begin{theorem}[Special case of~\cite{fund-limits-wigner}, informal]
\label{thm:main-informal-stat}
Consider the spiked Wigner testing problem (Definition~\ref{def:wigner}) where $\pi$ is the Rademacher prior (uniform on $\{-1,+1\}$) and $\lambda < 1$. No test (regardless of runtime) has error tradeoff asymptotically better than LSS.
\end{theorem}

\noindent The precise statement of this result is presented in Section~\ref{sec:our-approach-stat}, particularly Theorem~\ref{thm:main-stat}. The proof of~\cite{fund-limits-wigner} characterizes the limiting distribution of the likelihood ratio by leveraging some powerful machinery of Guerra and Talagrand from spin-glass theory. The same result can be deduced from the earlier work~\cite[Proposition~2.2]{ALR}, which is specific to the Rademacher prior and takes a combinatorial approach called \emph{cluster expansion}. Our proof is quite different, and arguably more elementary. We require two ingredients: first, a sharp bound on the second moment of the likelihood ratio; and second, a positive result showing achievability of some tradeoff curve (between errors of types I and II) that ``saturates'' the first bound. In the case of Theorem~\ref{thm:main-informal-stat}, the positive result is the analysis of LSS from~\cite{weak-wigner}. In a nutshell, the standard approach requires a direct analysis of the likelihood ratio, whereas our approach requires the analysis of \emph{any} test combined with a matching second moment bound.

\subsection*{Notation}

By default, asymptotic notation refers to the limit $n \to \infty$ with all other parameters held fixed (aside from those that are explicitly allowed to scale with $n$, such as $D = D_n$). In other words, notation such as $O(\cdot)$, $\Omega(\cdot)$, $o(\cdot)$, $\omega(\cdot)$ may hide factors depending on constants such as $\lambda,\pi$. We use $\poly(n)$ as shorthand for $n^{O(1)}$, and $\polylog(n)$ as shorthand for $(\log n)^{O(1)}$. The term ``polynomial time'' refers to an algorithm of runtime $\poly(n)$.

\section{Main Results}

\subsection{Background}
\label{sec:background}

\subsubsection{ROC curve}

We recall some basic notions from the theory of hypothesis testing, referring the reader to~\cite{testing-book} for a standard reference. Given two distributions $\cP,\cQ$ on a set $\Omega$, we consider testing the null hypothesis $Y \sim \cQ$ against the (simple) alternative $Y \sim \cP$. (Some authors use the opposite meaning of $\cP$ and $\cQ$ but we use the mnemonic $\cP$ = ``planted.'') A \emph{test} is a (possibly random) function $t: \Omega \to \{\fp,\fq\}$, where the symbols $\fp,\fq$ encode the assertion that $Y$ was drawn from $\cP$ or $\cQ$ respectively. The \emph{size} $\alpha \in [0,1]$ of a test (also called the \emph{type I error rate} or \emph{false positive rate}) is defined as
\begin{equation}\label{eq:def-alpha}
\alpha = \cQ(t(Y) = \fp),
\end{equation}
and the \emph{power} $\beta \in [0,1]$ (or \emph{true positive rate}) is defined as
\begin{equation}\label{eq:def-beta}
\beta = \cP(t(Y) = \fp).
\end{equation}
(Here we follow the convention of~\cite{testing-book} but note that some authors use $\beta$ for the type II error rate, which is $1-\beta$ in our notation.) Constrained to a given value of $\alpha$, it is desirable to find a test maximizing $\beta$.

For a class of tests $\cC$, let $N = N_\cC \subseteq [0,1]^2$ denote the set of $(\alpha,\beta)$ pairs for which there exists a test $t \in \cC$ satisfying~\eqref{eq:def-alpha} and~\eqref{eq:def-beta}. There are trivial tests achieving the points $(0,0)$ and $(1,1)$, so these are contained in $N$ (assuming the class $\cC$ contains these trivial tests). Also, by considering probabilistic mixtures of two tests (i.e., run $t_1$ with probability $p$ and $t_2$ with probability $1-p$), it is clear that $N$ is a convex set (assuming $\cC$ is closed under probabilistic mixtures). Finally, by flipping the output of a test, $N$ is symmetric with respect to the point $(1/2,1/2)$ in the sense that $(\alpha,\beta) \in N$ if and only if $(1-\alpha,1-\beta) \in N$ (again assuming $\cC$ is closed under flipping the output). We therefore restrict our attention to $(\alpha,\beta)$ pairs in the upper triangle
\[ \Delta := \{(\alpha,\beta) \in [0,1]^2 \,:\, \alpha \le \beta\}. \]
The curve $\phi: [0,1] \to [0,1]$ that bounds the upper edge of the region $N$, namely
\[ \phi(\alpha) = \sup\{\beta \,:\, (\alpha,\beta) \in N\}, \]
is called the \emph{ROC (receiver operating characteristic) curve} for the class $\cC$. The properties of $N$ discussed above imply that $\phi$ is increasing and concave, with $\phi(\alpha) \ge \alpha$. This curve describes the possible tradeoffs between type I and type II errors achievable by $\cC$. At one extreme, $\cC$ could be the class of tests that thresholds a particular statistic at varying thresholds. At the other extreme, $\cC$ could be the class of \emph{all} tests, in which case the corresponding \emph{optimal} ROC curve describes the fundamental limits for testing error.

\begin{remark}[Terminology]
The ROC curve is a popular notion in machine learning. A related notion from classical statistics is the \emph{power function} of a test, although typically this describes the power as a function of the alternative hypothesis (e.g., the parameter $\lambda$ in the spiked Wigner model), rather than a function of $\alpha$. Similarly, the optimal ROC curve (when $\cC$ is the class of all tests) is related to the notion of \emph{power envelope} (or \emph{envelope power function}).
\end{remark}

We will be interested in an asymptotic setting where $\cP = \cP_n$ and $\cQ = \cQ_n$ are sequences of distributions, namely the spiked Wigner distributions $\cQ_n = \PP_{0,\pi,n}$ and $\cP_n = \PP_{\lambda,\pi,n}$ for a fixed choice of $\lambda,\pi$. Further, we will often restrict ourselves to tests $t = t_n$ with a given runtime $\cT_n$, where $\cT_n$ may stand for a class of asymptotic runtimes such as $O(n)$ or $\poly(n)$.

\begin{definition}\label{def:achieve}
Fix sequences $\cP_n$ (alternative) and $\cQ_n$ (null), and a runtime bound $\cT_n$. A point $(\alpha,\beta) \in \Delta$ is \emph{asymptotically achievable} in time $\cT_n$ if there is a sequence of tests $t = t_n$ computable in time $\cT_n$ such that
\[ \cQ_n(t_n(Y) = \fp) \le \alpha + o(1) \qquad\text{ and }\qquad \cP_n(t_n(Y) = \fp) \ge \beta - o(1) \]
as $n \to \infty$.
\end{definition}

\noindent We emphasize that $\alpha,\beta$ do not depend on $n$. For given sequences $\cP_n, \cQ_n$ and a given runtime bound $\cT_n$, one can define the set of asymptotically achievable $(\alpha,\beta)$ pairs analogous to $N$, as well as the associated asymptotic ROC curve.

\begin{remark}[Model of computation]
\label{rem:comp-model}
To rigorously define the notion of runtime, we must specify a model of computation for real-valued inputs. The arguments we use in this paper are not particularly sensitive to the model of computation, and for our purposes we need a model that is powerful enough to implement the LSS test in polynomial time (Theorem~\ref{thm:positive}) but weak enough that we believe the strong low-degree conjecture (Conjecture~\ref{conj:new}). The PhD thesis of Hopkins~\cite{hopkins-thesis}, which introduced the original low-degree conjecture, proposes (in a footnote on pg.\ 80) to use the real RAM model. For concreteness, one can take the real RAM model as the model of computation throughout this paper, specifically the variant defined by Blum--Shub--Smale~\cite{BSS}. The proof of Theorem~\ref{thm:positive} will include an explanation of how to implement the LSS test in this model. The real RAM model is an abstract model of computation where each memory cell can hold a real number, and the operations of addition, subtraction, multiplication, division, and comparison can be performed in a single step. This is arguably an unrealistic model for real-world digital computers because issues of numerical precision are ignored, but it is a convenient and popular theoretical abstraction. Plus, we are ultimately proving a negative result, which only becomes stronger when using a stronger model of computation. As a final note, we will allow algorithms to depend on the known parameters $\lambda,\pi,\alpha,\beta$ (rather than, say, taking these as input), which means that any real-valued constants depending on these parameters may be ``hard-coded'' into the algorithm.
\end{remark}

\subsubsection{Likelihood ratio and second moment}

Again consider sequences of distributions $\cP_n,\cQ_n$ on a set $\Omega_n$, and further assume $\cP_n$ is absolutely continuous with respect to $\cQ_n$ for each $n$. The \emph{likelihood ratio} $L_n = L_n(Y)$ is defined to be the Radon--Nikodym derivative $d\cP_n/d\cQ_n$. Working in the function space $\cL^2(\cQ_n)$ with inner product $\langle f,g \rangle := \EE_{Y \sim \cQ_n}[f(Y) \cdot g(Y)]$ and norm $\|f\| := \sqrt{\langle f,f \rangle}$, an important quantity is the \emph{squared norm (or second moment) of the likelihood ratio}:
\[ \|L_n\|^2 = \E_{Y \sim \cQ_n} [L_n(Y)^2] = \E_{Y \sim \cP_n} [L_n(Y)]. \]
It is a standard fact that asymptotic bounds on $\|L_n\|$ as $n \to \infty$ have implications for statistical indistinguishability of $\cP_n$ and $\cQ_n$ (see e.g.~\cite[Lemma~2]{MRZ}), namely:
\begin{itemize}
    \item If $\|L_n\| = O(1)$ then strong detection is impossible, or in other words, $(\alpha,\beta) = (0,1)$ is not asymptotically achievable (by any test, regardless of runtime).
    \item If $\|L_n\| = 1+o(1)$ then weak detection is impossible, or in other words, no $(\alpha,\beta) \in \Delta$ with $\alpha < \beta$ is asymptotically achievable.
\end{itemize}
(We always have $\|L_n\| \ge 1$, using Jensen's inequality and the fact $\EE_{\cQ_n}[L_n] = 1$.) This gives a powerful tool for ruling out strong or weak detection via a relatively tractable second moment calculation, which is carried out for the spiked Wigner model in~\cite{MRZ,BMVVX,opt-subopt}.

However, to exactly pin down the optimal asymptotic ROC curve in a regime where weak (but not strong) detection is possible, a more refined strategy is needed. The standard approach is the following. The Neyman--Pearson lemma implies that the optimal ROC curve is swept out by tests that threshold $L_n(Y)$, as the threshold is varied. It therefore suffices to determine the limiting distribution of $L_n$ under both $\cP_n$ and $\cQ_n$. This is often shown directly for one of the two hypotheses ($\cP_n$ or $\cQ_n$), and then the distribution of $L_n$ under the other can be deduced immediately via Le Cam's third lemma. This approach is carried out for the spiked Wigner model in~\cite{fund-limits-wigner}.

\subsubsection{Low-degree testing}

Luckily, some parts of the above theory have natural analogues in the setting where we restrict our attention to computationally efficient tests --- or rather, to low-degree polynomial tests as a proxy for this~\cite{HS-bayesian,sos-detect,hopkins-thesis} (see also the survey~\cite{ld-notes}). We consider the same asymptotic setting as above and additionally assume the domain $\Omega_n$ is a subset of $\RR^M$ for some $M = M_n$ so that we may speak of (multivariate) polynomial functions $f: \Omega_n \to \RR$. The analogue of the likelihood ratio $L$ (suppressing $n$-dependence for ease of notation) is the \emph{low-degree likelihood ratio} $L^{\le D}$, which is the orthogonal projection in $\cL^2(\cQ)$ of $L$ onto the subspace of degree-$D$ polynomial functions. The norm of $L^{\le D}$ plays a similar role as its statistical analogue (see e.g.~\cite[Proposition~6.2]{fp}), namely:
\begin{itemize}
    \item If $\|L^{\le D}\| = O(1)$ for some $D = D_n$ then no degree-$D$ polynomial $f = f_n$ achieves \emph{strong separation} between $\cP$ and $\cQ$, defined as
    \[ \sqrt{\max\left\{\V_{\cP}[f],\V_{\cQ}[f]\right\}} = o\left(\left|\E_\cP[f] - \E_\cQ[f]\right|\right) \]
    as $n \to \infty$. Note that strong separation is a natural sufficient condition for strong detection by thresholding $f(Y)$.
    \item If $\|L^{\le D}\| = 1+o(1)$ for some $D = D_n$ then no degree-$D$ polynomial $f = f_n$ achieves \emph{weak separation} between $\cP$ and $\cQ$. Weak separation has the same definition as strong separation but with $O(\cdots)$ in place of $o(\cdots)$, and is a natural sufficient condition for weak detection using the value of $f(Y)$~\cite[Proposition~6.1]{fp}.
\end{itemize}

\noindent Results of this form are considered ``evidence'' that the associated strong/weak detection problem is inherently hard for algorithms of the corresponding runtime.

\begin{conjecture}[Low-degree conjecture, informal]
\label{conj:ld}
If degree-$D$ polynomials fail to solve a testing problem (in the sense of strong/weak separation) for some $D = \omega(\log n)$, then there is no polynomial-time algorithm for the associated strong/weak (respectively) detection task. In general, if degree-$D$ polynomials fail then there is no algorithm of runtime $\exp(\tilde\Omega(D))$ where $\tilde\Omega(\cdot)$ hides a $\polylog(n)$ factor.
\end{conjecture}

\noindent This informal conjecture is inspired by~\cite[Hypothesis~2.1.5 \& Conjecture~2.2.4]{hopkins-thesis}, along with the fact that low-degree polynomials capture the best known algorithms for a wide variety of statistical tasks. The conjecture appears to hold up for distributions $\cP,\cQ$ of a particular style that often arises in high-dimensional statistics (including the spiked Wigner and Wishart models~\cite{sk-cert,ld-notes,subexp-sparse,spectral-planting}), but it does not hold for \emph{all} distributions $\cP,\cQ$, and we refer the reader to~\cite{hopkins-thesis,ld-notes,HW-counter,morris,KM-tree,lattice-1,lattice-2} for further discussion on which distributions are appropriate. Given the current state of average-case complexity theory, we do not realistically hope to prove any variant of the low-degree conjecture. Rather, its purpose is to serve as a guide for making principled conjectures about average-case complexity in settings where we would otherwise have no way to make progress.

The above framework provides a useful tool for probing the computational feasibility of strong or weak detection, but our goal in this work is to ask a more refined question: we seek the best possible ROC curve achievable by algorithms of a given runtime, in a regime where weak (but not strong) detection is possible by this class of algorithms. As discussed previously, the statistical analogue of this question is traditionally attacked via the Neyman--Pearson lemma combined with Le Cam's third lemma. This approach does not seem viable in the computationally bounded setting because there is no analogue of Neyman--Pearson, that is, there is no guarantee that thresholding $L^{\le D}$ achieves the best possible ROC curve among low-degree tests. We will overcome this by taking a different approach, explained in the following sections.

\subsection{Our Approach: Statistical Limits}
\label{sec:our-approach-stat}

To explain our approach, we first consider the purely statistical question of determining the optimal ROC curve with no constraints on runtime. This section serves as a ``warm-up'' to the more general framework presented in the next section, which accounts for computational complexity. The ideas in this section may also be of independent interest, as they provide a way to indirectly characterize the optimal ROC curve without analyzing the distribution of the likelihood ratio.

We will illustrate our approach by re-proving the following known result in the spiked Wigner model. For $\lambda > 0$ and $\alpha \in [0,1]$, define
\begin{equation}\label{eq:phi_lambda}
\phi_\lambda(\alpha) := 1 - \Phi\left[\Phi^{-1}(1-\alpha) - \sqrt{\frac{1}{2} \log\left(\frac{1}{1-\lambda^2}\right)}\right]
\end{equation}
where $\Phi$ denotes the standard normal CDF function, and we take the conventions $\phi_\lambda(0) = 0$, $\phi_\lambda(1) = 1$. This is the ROC curve achieved by linear spectral statistics (LSS)~\cite{weak-wigner}. We will show this is statistically optimal for the Rademacher spike prior.

\begin{theorem}[Special case of~\cite{fund-limits-wigner}]
\label{thm:main-stat}
Consider the spiked Wigner testing problem (Definition~\ref{def:wigner}) where $\pi$ is uniform on $\{-1,+1\}$ and $\lambda \in (0,1)$. Any $(\alpha,\beta) \in [0,1]^2$ with $\beta > \phi_\lambda(\alpha)$ is not asymptotically achievable (in the sense of Definition~\ref{def:achieve}), regardless of runtime.
\end{theorem}

\noindent Since we are considering the Rademacher prior, the model lacks spherical symmetry, which precludes approaches such as~\cite{sphericity,BL-free-1,BL-free-2}. The proof of~\cite{fund-limits-wigner} uses machinery developed by Guerra and Talagrand in their study of the Sherrington--Kirkpatrick spin-glass model, and the same result also follows from~\cite[Proposition~2.2]{ALR}, which takes a combinatorial approach specific to the Rademacher prior. Our alternate proof instead uses the analysis of LSS from~\cite{weak-wigner}, combined with a relatively simple second moment calculation from~\cite{opt-subopt}.

We begin with the second moment calculation.

\begin{theorem}[Special case of~\cite{opt-subopt}, Theorem~3.10]
\label{thm:L-bound-rad}
Consider the spiked Wigner testing problem (Definition~\ref{def:wigner}) where $\pi$ is uniform on $\{-1,+1\}$ and $\lambda \in (0,1)$. The likelihood ratio $L_\lambda := d\PP_\lambda/d\PP_0$ has $\cL^2(\PP_0)$-norm (as defined in Section~\ref{sec:background})
\begin{equation}\label{eq:L-bound-rad}
\lim_{n \to \infty} \|L_\lambda\| = (1-\lambda^2)^{-1/4}.
\end{equation}
\end{theorem}

\noindent The utility of this result stems from the following well-known operational definition for the norm of the likelihood ratio, which is a special case of~\eqref{eq:op-def} below:
\begin{equation}\label{eq:op-def-wig}
\|L_\lambda\| = \sup_f \frac{\EE_{Y \sim \PP_\lambda}[f(Y)]}{\sqrt{\EE_{Y \sim \PP_0}[f(Y)^2]}},
\end{equation}
where the supremum is over all real-valued functions $f \in \cL^2(\PP_0)$. The fact~\eqref{eq:op-def-wig} on its own places some constraints on what the optimal ROC curve can look like. For instance, the relation
\begin{equation}\label{eq:abL}
\frac{\beta^2}{\alpha} + \frac{(1-\beta)^2}{1-\alpha} \le \|L_\lambda\|^2
\end{equation}
must be satisfied by every achievable pair $(\alpha,\beta) \in (0,1)^2$~\cite[Proposition~2.5]{opt-subopt}. This is a weaker condition than the desired one, $\beta \le \phi_\lambda(\alpha)$. However, the information we get from~\eqref{eq:op-def-wig} is more subtle than~\eqref{eq:abL} alone. At this point there remain many viable candidates for the optimal ROC curve --- one of which is $\phi_\lambda$ --- that are incomparable to each other. To illustrate, each $(\alpha,\beta)$ pair satisfying~\eqref{eq:abL} with equality is achieved by one such viable ROC curve, namely a straight line from $(0,0)$ to $(\alpha,\beta)$ followed by a straight line from $(\alpha,\beta)$ to $(1,1)$. However, it is \emph{not} possible that \emph{all} the $(\alpha,\beta)$ pairs satisfying~\eqref{eq:abL} are simultaneously achieved by the true ROC curve, or else we could construct a function $f$ that yields a too-good-to-be-true value for the ratio in~\eqref{eq:op-def-wig}. We will use the positive result on the achievability of $\phi_\lambda$ to disambiguate between these candidate ROC curves and conclude that $\phi_\lambda$ is in fact the true ROC curve.

Now in more detail: as we will show (see Section~\ref{sec:pf-overview}), the fact that $\phi_\lambda$ is achievable allows us to construct a function $f$ that makes the ratio in~\eqref{eq:op-def-wig} asymptotically equal to
\[ \val(\phi_\lambda) := \sqrt{\int_0^1 (\phi_\lambda'(\alpha))^2 \,d\alpha} = (1-\lambda^2)^{-1/4}, \]
saturating~\eqref{eq:L-bound-rad}. Any hypothetical improvement to $\phi_\lambda$, even at a single point, would allow us to construct an even better function $f$, leading to an even larger value for the ratio, contradicting~\eqref{eq:L-bound-rad}. This is formalized in Proposition~\ref{prop:reduction-stat} below and illustrated in Figure~\ref{fig:pf}.

\begin{proposition}[Special case of Proposition~\ref{prop:reduction}]
\label{prop:reduction-stat}
Given sequences of distributions $\cP = \cP_n$ and $\cQ = \cQ_n$ on $\Omega = \Omega_n$, consider testing the null hypothesis $Y \sim \cQ$ against the alternative $Y \sim \cP$. For some $\lambda \in (0,1)$, suppose we have the following positive result: for every $\alpha,\beta \in [0,1]$ with $\alpha \le \beta < \phi_\lambda(\alpha)$, there is a test that asymptotically achieves $(\alpha,\beta)$, in the sense of Definition~\ref{def:achieve}. Suppose further there exists some $(\alpha^*,\beta^*) \in [0,1]^2$ with $\beta^* > \phi_\lambda(\alpha^*)$ that is also asymptotically achievable. Then there exists a function $f = f_n: \Omega_n \to \RR$ such that
\begin{equation}\label{eq:ratio-goal-wig}
\liminf_{n \to \infty} \frac{\EE_{Y \sim \cP}[f(Y)]}{\sqrt{\EE_{Y \sim \cQ}[f(Y)^2]}} > \val(\phi_\lambda).
\end{equation}
\end{proposition}

\noindent Our desired conclusion now follows by combining the ingredients above.

\begin{proof}[Proof of Theorem~\ref{thm:main-stat}]
We will apply Proposition~\ref{prop:reduction-stat} with $\cP = \PP_\lambda$ and $\cQ = \PP_0$. The required positive result follows from the LSS analysis of~\cite{weak-wigner}, which is stated formally as Theorem~\ref{thm:positive} in the next section. Assume on the contrary that some $(\alpha^*,\beta^*)$ with $\beta^* > \phi_\lambda(\alpha^*)$ is asymptotically achievable. By Proposition~\ref{prop:reduction-stat} we obtain $f$ such that
\[ \liminf_{n \to \infty} \frac{\EE_{Y \sim \PP_\lambda}[f(Y)]}{\sqrt{\EE_{Y \sim \PP_0}[f(Y)^2]}} > \val(\phi_\lambda), \]
and the computation $\val(\phi_\lambda) = (1-\lambda^2)^{-1/4}$ is Lemma~\ref{lem:calc-val} in the next section. This now contradicts~\eqref{eq:L-bound-rad},\eqref{eq:op-def-wig}.
\end{proof}

The next section contains a more general form of this argument that also considers computational complexity. The general form of Proposition~\ref{prop:reduction-stat} (i.e., Proposition~\ref{prop:reduction}) applies to a wide class of curves $\phi$ rather than simply $\phi_\lambda$. We note, however, that ROC curves of the form $\phi_\lambda$ are ``common'' in that they arise from thresholding a statistic whose limiting distributions under null and alternative are two different Gaussians with the same variance.

For simplicity, we have focused here on recovering the result of~\cite{fund-limits-wigner} for one particular spike prior --- Rademacher. The main obstacle to recovering the full result of~\cite{fund-limits-wigner} for other priors $\pi$ is establishing $\|L_\lambda\| = O(1)$ for all $\lambda$ below the critical threshold $\lambda^*(\pi)$. For some priors this is simply not true and would need to be replaced with a conditional second moment calculation. See~\cite{BMVVX,opt-subopt} for some results of this form, albeit only reaching the sharp threshold $\lambda^*(\pi)$ for certain priors.

\begin{remark}\label{rem:non-GOE}
Theorem~\ref{thm:main-stat} can be generalized to a modified spiked Wigner testing problem where the noise matrix $W$ has diagonal entries $W_{ii} \sim \cN(0,\sigma^2/n)$ instead of $\cN(0,2/n)$ for a constant $\sigma > 0$, and everything else is unchanged. This case is covered by~\cite{fund-limits-wigner}, and our method also extends to this case in a straightforward way. As in the proof of Theorem~\ref{thm:positive}, the results of~\cite{weak-wigner} imply that LSS asymptotically achieves the ROC curve $\phi_{\lambda,\sigma}(\alpha) := 1 - \Phi\left[\Phi^{-1}(1-\alpha) - \sqrt{\mu/2}\right]$ with $\mu := -\log(1-\lambda^2) + \lambda^2(2-\sigma^2)/\sigma^2$. As in the proof of Lemma~\ref{lem:calc-val}, $\val(\phi_{\lambda,\sigma}) = \exp(\mu/4) = (1-\lambda^2)^{-1/4} \exp(\lambda^2(2-\sigma^2)/(4\sigma^2))$. For the analysis of the likelihood ratio $L_{\lambda,\sigma}$, we have the standard formula (see Lemma~1 of~\cite{BMVVX})
\[ \|L_{\lambda,\sigma}\|^2 = \E_{x,x'} \exp\left(\frac{\lambda^2 n}{2} \frac{\langle x,x' \rangle^2}{\|x\|^2 \|x'\|^2} + \lambda^2 n \left(\frac{1}{\sigma^2}-\frac{1}{2}\right) \sum_{i=1}^n \frac{x_i^2 (x_i')^2}{\|x\|^2 \|x'\|^2}\right) \]
where $x$ is i.i.d.\ Rademacher and $x'$ is an independent copy of $x$. This differs from the original second moment $\|L_\lambda\|^2$ due to the additional second term in $\exp(\cdots)$, which (due to the Rademacher prior) is equal to the constant $\lambda^2 (2-\sigma^2)/(2\sigma^2)$, i.e., $\|L_{\lambda,\sigma}\|^2 = \|L_\lambda\|^2 \exp(\lambda^2 (2-\sigma^2)/(2\sigma^2)) \to (1-\lambda^2)^{-1/2} \exp(\lambda^2 (2-\sigma^2)/(2\sigma^2))$. Therefore, the limiting value of $\|L_{\lambda,\sigma}\|$ matches $\val(\phi_{\lambda,\sigma})$. The conclusion for the modified model is identical to Theorem~\ref{thm:main-stat} except the optimal ROC curve is $\phi_{\lambda,\sigma}$ in place of $\phi_\lambda$.
\end{remark}

\tikzset{
    declare function={
        normcdf(\x)=1/(1 + exp(-0.07056*(\x)^3 - 1.5976*(\x)));
    }
}

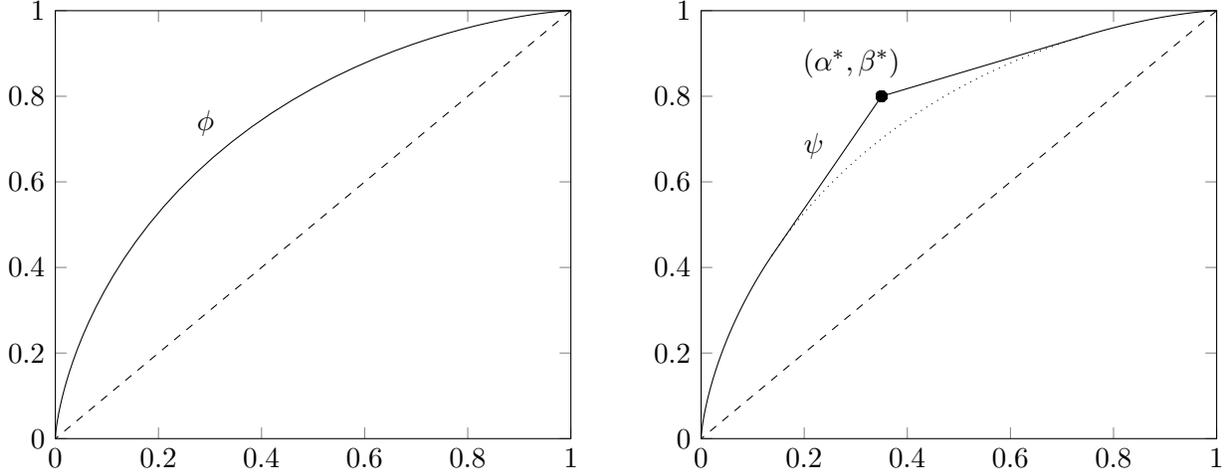
\begin{figure}
\centering

\begin{subfigure}[b]{0.48\textwidth}
    \centering
    \def\L{0.9}
    \def\B{5}
    \def\samp{100}
    \def\m{sqrt(ln(1/(1-(\L)^2))/8)}
    \begin{tikzpicture}
        \begin{axis}[xmin=0,xmax=1,ymin=0,ymax=1,
        legend pos=south east]
        \addplot [mark=none, black, dashed, samples=2] {x};
        \addplot [domain=-\B:\B, samples=\samp, black]({1-normcdf(x+\m)},{1-normcdf(x-\m)});
        \end{axis}
    \node at (2,4.2) {$\phi$};
    \end{tikzpicture}
\end{subfigure}
\hfill
\begin{subfigure}[b]{0.48\textwidth}
    \centering
    \def\L{0.9}
    \def\B{5}
    \def\samp{100}
    \def\aa{0.35}
    \def\bb{0.8}
    \def\tt{-1.1}
    \def\ss{0.65}
    \def\m{sqrt(ln(1/(1-(\L)^2))/8)}
    \begin{tikzpicture}
        \begin{axis}[xmin=0,xmax=1,ymin=0,ymax=1,
        legend pos=south east]
        \addplot [mark=none, black, dashed, samples=2] {x};
        \addplot [domain=\tt:\ss, samples=\samp, black, dotted]({1-normcdf(x+\m)},{1-normcdf(x-\m)});
        \addplot [domain=-\B:\tt, samples=\samp, black]({1-normcdf(x+\m)},{1-normcdf(x-\m)});
        \addplot [domain=\ss:\B, samples=\samp, black]({1-normcdf(x+\m)},{1-normcdf(x-\m)});
        \addplot [only marks, mark=*](\aa,\bb);
        \addplot [domain=0:1, mark=none, black, samples=2]({x*(1-normcdf(\tt+\m))+(1-x)*\aa},{x*(1-normcdf(\tt-\m))+(1-x)*\bb});
        \addplot [domain=0:1, mark=none, black, samples=2]({x*(1-normcdf(\ss+\m))+(1-x)*\aa},{x*(1-normcdf(\ss-\m))+(1-x)*\bb});
        \end{axis}
    \node at (2,5) {$(\alpha^*,\beta^*)$};
    \node at (1.5,3.9) {$\psi$};
    \end{tikzpicture}
\end{subfigure}

\caption{Illustration of the proof of Theorem~\ref{thm:main-stat}. {\bf Left}: The curve $\phi = \phi_\lambda$ for $\lambda = 0.9$. All points $(\alpha,\beta)$ below the curve $\phi$ (and above the ``trivial'' dashed line) are achievable. {\bf Right}: If hypothetically there were an achievable point $(\alpha^*,\beta^*)$ above the curve $\phi$ then every point below $\psi$ --- the upper concave envelope of $\phi$ and $(\alpha^*,\beta^*)$ --- would also be achievable. The improved curve $\psi$ would allow us to construct a function $f$ that makes the ratio in~\eqref{eq:op-def-wig} equal to $\val(\psi)$, which is strictly greater than $\val(\phi_\lambda) = (1-\lambda^2)^{-1/4}$, contradicting~\eqref{eq:L-bound-rad}.}
\label{fig:pf}
\end{figure}

To what extent is our new proof of Theorem~\ref{thm:main-stat} simpler or more broadly applicable than the existing proofs of~\cite{ALR,fund-limits-wigner}? Certainly the second moment bound (Theorem~\ref{thm:L-bound-rad}) is more elementary than the interpolation methods used in~\cite{fund-limits-wigner}, and it also seems more broadly applicable than the combinatorial approach of~\cite{ALR} (see~\cite{opt-subopt}). On the other hand, our argument also relies on the LSS analysis of~\cite{weak-wigner}, which is rather involved (although quite different from~\cite{ALR,fund-limits-wigner}). At the very least, our approach gives the flexibility to trade off one argument for another, and if one is planning to do the LSS analysis anyway then the matching impossibility result is obtained essentially for free. Perhaps the greatest strength of our approach is its ability to generalize to the computationally bounded setting, which we discuss next.

\subsection{Our Approach: Computational Limits}
\label{sec:our-approach-comp}

Building on ideas from the previous section, we now outline our approach for arguing that no (computationally) efficient algorithm can beat the ROC curve of LSS in the spiked Wigner model. As this is a statement about average-case complexity, there will inevitably need to be a conjecture involved (Conjecture~\ref{conj:new}). Our final conclusion, conditional on the conjecture, is Corollary~\ref{cor:main-comp}. The proofs of the results in this section are deferred to Section~\ref{sec:proofs}.

Throughout this section, we consider the spiked Wigner testing problem described in Definition~\ref{def:wigner}: we are testing the null hypothesis $Y \sim \PP_0$ against a specific alternative $Y \sim \PP_\lambda$, in the regime $n \to \infty$ with $\lambda, \pi$ fixed. We place the following mild assumptions on the spike prior $\pi$.

\begin{assumption}\label{assum:pi}
The spike prior $\pi$ is a probability distribution on $\RR$ such that:
\begin{itemize}
\item (Centered) $\EE[\pi] = 0$,
\item (Unit variance) $\EE[\pi^2] = 1$,
\item (Subgaussian) there exists a constant $c > 0$ such that $\EE[\exp(t\pi)] \le \exp(ct^2)$ for all $t \in \RR$. (This is satisfied by any distribution with bounded support, by Hoeffding's lemma.)
\end{itemize}
\end{assumption}

\noindent While we focus on the spiked Wigner model here, it will be clear that our approach can potentially be applied more generally.

\subsubsection{Step 1: Guess the ROC curve and prove a positive result}

The first step in our approach is to ``guess'' the ROC curve that we aim to show is optimal among efficient algorithms. While our final goal is to show a hardness result, it will (perhaps counterintuitively) be important to prove the corresponding positive result, that is, our ROC curve is achievable by an efficient algorithm. In our case, we choose the ROC curve~\eqref{eq:phi_lambda} that is known to be achievable by LSS, which we repeat here for convenience:
\begin{equation}\label{eq:phi_lambda-2}
\phi_\lambda(\alpha) := 1 - \Phi\left[\Phi^{-1}(1-\alpha) - \sqrt{\frac{1}{2} \log\left(\frac{1}{1-\lambda^2}\right)}\right]
\end{equation}
where $\Phi$ denotes the standard normal CDF function, and we take the conventions $\phi_\lambda(0) = 0$, $\phi_\lambda(1) = 1$. We now state the corresponding positive result, which will be extracted from~\cite{weak-wigner} in a straightforward way; the details of the proof are deferred to Appendix~\ref{app:pf-thm-positive}.

\begin{theorem}
\label{thm:positive}
Fix any $\lambda \in (0,1)$, any spike prior $\pi$, and any $\alpha,\beta \in [0,1]$ with $\alpha \le \beta \le \phi_\lambda(\alpha)$. There is a polynomial-time algorithm that asymptotically achieves $(\alpha,\beta)$ for the spiked Wigner testing problem (in the sense of Definition~\ref{def:achieve}).
\end{theorem}

We have defined spike priors to have mean 0 and variance 1 throughout the paper, but Theorem~\ref{thm:positive} holds with no assumptions on $\pi$ (as long as $x \ne 0$ with probability $1-o(1)$). Due to GOE rotational invariance, the LSS test is not affected by the direction of the spike.

\subsubsection{Step 2: Matching bound on $\|L^{\le D}\|$}

A key step in our approach is to establish a sharp bound on the norm of the low-degree likelihood ratio (defined in Section~\ref{sec:background}) that matches the ``value'' of our ROC curve, defined as follows.

\begin{definition}
For a function $\phi: [0,1] \to [0,1]$ whose derivative $\phi'$ exists on $(0,1)$ except at finitely many points, define
\begin{equation}\label{eq:def-val}
\val(\phi) = \sqrt{\int_0^1 (\phi'(\alpha))^2 \,d\alpha},
\end{equation}
provided the (improper Riemann) integral exists.
\end{definition}

\noindent The intuition for this value is the following fact, which will be implicit in the proof of Proposition~\ref{prop:reduction} and is discussed further in Section~\ref{sec:pf-overview}: if some (concave) $\phi$ is achievable then we can construct a function $f$ that makes a particular ratio --- namely $R_\lambda(f)$ defined in~\eqref{eq:def-R} below --- equal to $\val(\phi)$.

\begin{lemma}\label{lem:calc-val}
For the function $\phi_\lambda$ defined in~\eqref{eq:phi_lambda-2},
\[ \val(\phi_\lambda) = (1-\lambda^2)^{-1/4}. \]
\end{lemma}

\noindent The details of the above calculation are deferred to Appendix~\ref{app:pf-lem-calc-val}.

\begin{theorem}\label{thm:L-bound}
Fix any $\lambda \in (0,1)$ and any $\pi$ satisfying Assumption~\ref{assum:pi}. Let $L_\lambda = L_{\lambda,\pi,n}$ denote the spiked Wigner likelihood ratio $d\PP_\lambda/d\PP_0$. Let $D = D_n$ be any sequence satisfying $D = \omega(1)$ and $D = o(n/\log n)$. The $\cL^2(\PP_0)$-norm of the degree-$D$ likelihood ratio (as defined in Section~\ref{sec:background}) satisfies
\[ \lim_{n \to \infty} \|(L_\lambda)^{\le D}\| = (1-\lambda^2)^{-1/4}. \]
\end{theorem}

\noindent Crucially, the value on the right-hand side above matches $\val(\phi_\lambda)$. The assumption $D = \omega(1)$ is only needed for the lower bound, i.e., the upper bound $\|(L_\lambda)^{\le D}\| \le (1-\lambda^2)^{-1/4} + o(1)$ still holds without this assumption. The proof of Theorem~\ref{thm:L-bound} builds on some existing tools for similar second moment calculations~\cite{opt-subopt,sk-cert,ld-notes,spectral-planting,fp}.

Theorem~\ref{thm:L-bound} gives a surprisingly sharp phase transition: for a fixed $\lambda \in (0,1)$, the limiting value of $\|(L_\lambda)^{\le D}\|$ has no dependence on $D$ when $D$ ranges between $\omega(1)$ and $o(n/\log n)$. Once $D$ reaches order $n$, the value of $\|(L_\lambda)^{\le D}\|$ can (for some priors) suddenly jump to size $\exp(\Omega(n))$ as discussed in Remark~\ref{rem:large-D} below. This behavior is consistent with the fact that we do not know any algorithm of subexponential time $\exp(O(n^\delta))$ for a constant $\delta < 1$ that beats the ROC curve of LSS, but some priors do admit such an algorithm of exponential time $\exp(O(n))$.

\begin{remark}\label{rem:large-D}
Regarding the condition $D=o(n/\log n)$ in Theorem~\ref{thm:L-bound}: while the log factor might be an artifact of the proof, it is necessary to assume $D = o(n)$, or else $\|(L_\lambda)^{\le D}\|$ can diverge to infinity (for some choices of $\lambda \in (0,1)$ and $\pi$). To see this, take $\pi$ to be the sparse Rademacher prior which takes values $\pm 1/\sqrt{\rho}$ each with probability $\rho/2$ and value $0$ with probability $1-\rho$, for a small constant $\rho > 0$ to be chosen later. With $x \in \RR^n$ having entries i.i.d.\ from $\pi$ and $x'$ an independent copy of $x$, a standard formula~\eqref{eq:L-wigner} gives
\[ \|(L_\lambda)^{\le D}\|^2 = \sum_{d=0}^D \frac{1}{d!} \E_{x,x'} \left[\left(\frac{\lambda^2 n \langle x,x' \rangle^2}{2 \|x\|^2 \|x'\|^2}\right)^d\right] \ge \frac{1}{D!} \Pr\{x=x' \ne 0\} \left(\frac{\lambda^2 n}{2}\right)^D. \]
Using the bounds $\Pr\{x=x'\ne 0\} \ge \Pr\{x=x'=e_1/\sqrt{\rho}\} = (\rho/2)^2 (1-\rho)^{2(n-1)}$ and $D! \le D^D$, the above is at least $\exp[D \log(\lambda^2 n/(2D)) + 2n \log(1-\rho) + 2 \log (\rho/2)]$. Set $D = \lfloor \epsilon n \rfloor$ for a constant $\epsilon > 0$ small enough so that $\lambda^2/(2\epsilon) > 1$. Now choosing $\rho$ small enough (depending on $\lambda,\epsilon$), we have $\|(L_\lambda)^{\le D}\|^2 = \exp(\Omega(n))$.
\end{remark}

\subsubsection{Step 3: Strengthening of the low-degree conjecture}

As discussed in Section~\ref{sec:background}, the low-degree conjecture posits that (for certain testing problems) low-degree polynomials are at least as powerful as all algorithms of the corresponding runtime (where the correspondence is described in Conjecture~\ref{conj:ld}). For instance, in the spiked Wigner model with $\lambda < 1$, we have from Theorem~\ref{thm:L-bound} that $\|(L_\lambda)^{\le D}\| = O(1)$ for any $D = o(n/\log n)$, and we consider this to be evidence that strong detection requires runtime $\exp(n^{1-o(1)})$. It is well known (see~\cite{hopkins-thesis,ld-notes}) that the norm of the low-degree likelihood ratio admits the operational definition
\begin{equation}\label{eq:op-def}
\left\|\left(\frac{d\cP}{d\cQ}\right)^{\le D}\right\| = \sup_{f \,:\, \deg(f) \le D} \frac{\EE_{Y \sim \cP}[f(Y)]}{\sqrt{\EE_{Y \sim \cQ}[f(Y)^2]}},
\end{equation}
where the supremum is over polynomials $f: \Omega \to \RR$ of degree at most $D$. In fact, the supremum is achieved and the optimizer is $f = (d\cP/d\cQ)^{\le D}$. For our purposes, we introduce a natural refinement of the low-degree conjecture: we posit that low-degree polynomials perform at least as well as all algorithms of the corresponding runtime \emph{in terms of the value of the ratio} on the right-hand side of~\eqref{eq:op-def}. In the spiked Wigner case, this ratio is
\begin{equation}\label{eq:def-R}
R_\lambda(f) := \frac{\EE_{Y \sim \PP_\lambda}[f(Y)]}{\sqrt{\EE_{Y \sim \PP_0}[f(Y)^2]}}.
\end{equation}

\begin{remark}[Update on the conjecture]
The original version of this work contained a different conjecture, which has since been refuted by Ansh Nagda. The conjecture below has been modified to survive Nagda's counterexample. Our results continue to hold under this new conjecture. See Appendix~\ref{app:update} for further details.
\end{remark}

\begin{conjecture}\label{conj:new}
Consider the spiked Wigner testing problem (Definition~\ref{def:wigner}). Fix any $0 < \delta_1 < \delta_2$, any $\lambda \in (0,1)$, and any $\pi$ satisfying Assumption~\ref{assum:pi}. Consider any sequence of functions $f = f_n$ with the following properties:
\begin{itemize}
\item $f$ is computable in time $\exp(O(n^{\delta_1}))$, and
\item there exists a constant $B \ge 1$ (not depending on $n$) such that $f(Y) \in [1/B,B]$ for all $Y$.
\end{itemize}
Then,
\begin{equation}\label{eq:conj-ratio}
\limsup_{n \to \infty} R_\lambda(f) \le \limsup_{n \to \infty} \sup_{g \,:\, \deg(g) \le n^{\delta_2}} R_\lambda(g).
\end{equation}
\end{conjecture}

\noindent Recalling~\eqref{eq:op-def} and Theorem~\ref{thm:L-bound}, the specific case of interest for us will be the following: for any fixed $\epsilon > 0$, any $f = f_n$ computable in time $\exp(O(n^{1-\epsilon}))$ with bounded outputs must satisfy
\begin{equation}\label{eq:conj-specific}
\limsup_{n \to \infty} R_\lambda(f) \le (1-\lambda^2)^{-1/4}.
\end{equation}

To be conservative, we have stated Conjecture~\ref{conj:new} only for the spiked Wigner model, but one can imagine making similar conjectures in other settings. As with the standard low-degree conjecture we cannot realistically hope to prove Conjecture~\ref{conj:new}, but we can continue to amass evidence in its favor.

\begin{remark}
To justify Conjecture~\ref{conj:new}, one must believe that low-degree polynomials capture the most powerful approaches in our algorithmic toolkit. To this end, it may be instructive to see how the best known efficient algorithm for maximizing $R_\lambda(f)$ --- which uses LSS --- can be implemented as a low-degree polynomial. Recall the LSS statistic $H(Y) := \sum_{i=1}^n h_\lambda(\mu_i)$ from~\eqref{eq:lss}. After appropriate shifting and scaling, this is shown in~\cite{weak-wigner} to converge to a Gaussian limit: $a H(Y) + b \Rightarrow \cN(\pm c,1)$ where the plus sign holds under $\PP_\lambda$ and the minus sign holds under $\PP_0$, and $c := \sqrt{-\log(1-\lambda^2)/8}$. Now let $f(Y) = p(aH(Y)+b)$ and choose the optimal function $p$ to maximize $R_\lambda(f)$, which turns out to be the likelihood ratio between $\cN(c,1)$ and $\cN(-c,1)$, i.e., $p(z) = \exp(2cz)$. This yields $R_\lambda(f) \to (1-\lambda^2)^{-1/4}$, matching~\eqref{eq:conj-specific}. To approximate this with a polynomial, consider polynomial approximations $\hat{h}_\lambda$ and $\hat{p}$ for $h_\lambda$ and $p$, of degrees $D_1$ and $D_2$ respectively. Now $\hat{f}(Y) := \hat{p}(a\sum_i \hat{h}_\lambda(\mu_i) + b) = \hat{p}(a \cdot \Tr(\hat{h}_\lambda(Y)) + b)$ is a polynomial of degree $D = D_1 D_2$ that approximates $f$. We expect that any slowly growing degree $D = \omega(1)$ should suffice to construct such an $\hat{f}$ achieving $R_\lambda(\hat{f}) \to (1-\lambda^2)^{-1/4}$. While we have not attempted to rigorously analyze this polynomial approximation, we do know from Theorem~\ref{thm:L-bound} that \emph{some} degree-$\omega(1)$ polynomial achieves $R_\lambda \to (1-\lambda^2)^{-1/4}$.
\end{remark}

\subsubsection{Steps 1,2,3 are sufficient}

We now argue that the above steps are sufficient to deduce our desired conclusion. The idea is the following, which is similar to Figure~\ref{fig:pf}. Our achievable ROC curve $\phi_\lambda$ allows us to construct an efficiently computable function $f$ that achieves $R_\lambda(f) \approx \val(\phi_\lambda)$. If hypothetically there were an efficient algorithm achieving some point $(\alpha,\beta)$ \emph{above} the curve $\phi_\lambda$, this could be combined with $f$ to produce an efficiently computable $g$ that achieves $R_\lambda(g) > \val(\phi_\lambda)$, contradicting our refined low-degree conjecture.

We will state this part of the argument (Proposition~\ref{prop:reduction}) in high generality so that it can potentially be used for other problems beyond spiked Wigner in the future. We will require some conditions on the ROC curve $\phi$ (which in our case is $\phi_\lambda$). As discussed in Section~\ref{sec:background}, an ROC curve should be increasing and concave with $\phi(1) = 1$. We also impose some additional technical conditions (some of which are likely removable, but they allow for a cleaner proof).

\begin{assumption}\label{assum:phi}
Assume $\phi: [0,1] \to [0,1]$ has the following properties:
\begin{itemize}
\item $\phi(0) = 0$ and $\phi(1) = 1$,
\item $\phi$ is continuous on $[0,1]$ and differentiable on $(0,1)$,
\item $\phi'$ is continuous, strictly positive, and decreasing on $(0,1)$,
\item $\lim_{\alpha \to 0^+} \phi'(\alpha) = +\infty$ and $\lim_{\alpha \to 1^-} \phi'(\alpha) = 0$,
\item the integral~\eqref{eq:def-val} defining $\val(\phi)$ is finite.
\end{itemize}
\end{assumption}

\noindent Some immediate consequences are that $\phi$ is concave and strictly increasing, and $\phi(\alpha) \ge \alpha$ for all $\alpha \in [0,1]$.

\begin{lemma}\label{lem:check-phi-assum}
For any $\lambda \in (0,1)$, the function $\phi_\lambda$ defined in~\eqref{eq:phi_lambda-2} satisfies Assumption~\ref{assum:phi}.
\end{lemma}

\noindent The proof of Lemma~\ref{lem:check-phi-assum} is deferred to Appendix~\ref{app:pf-lem-check-phi-assum}.

The result below refers to a runtime bound $\cT$, which may stand for a class of asymptotic runtimes such as $O(n)$ or $\poly(n)$, or it may be $\infty$ (no bound on runtime). We will elaborate on this further in Remark~\ref{rem:runtime} below.

\begin{proposition}\label{prop:reduction}
Given sequences of distributions $\cP = \cP_n$ and $\cQ = \cQ_n$ on $\Omega = \Omega_n$, consider testing the null hypothesis $Y \sim \cQ$ against the alternative $Y \sim \cP$. Suppose we have a function $\phi$ satisfying Assumption~\ref{assum:phi} and a runtime bound $\cT = \cT_n$ along with the corresponding positive result: for every $\alpha,\beta \in [0,1]$ with $\alpha \le \beta < \phi(\alpha)$, there is an algorithm that asymptotically achieves $(\alpha,\beta)$ in time $\cT$, in the sense of Definition~\ref{def:achieve}. Suppose further there exists some $(\alpha^*,\beta^*) \in [0,1]^2$ with $\beta^* > \phi(\alpha^*)$ that is also asymptotically achievable in time $\cT$. Then there is a function $f = f_n: \Omega_n \to \RR$ computable in time $O(\cT)$ such that
\begin{equation}\label{eq:ratio-goal}
\liminf_{n \to \infty} \frac{\EE_{Y \sim \cP}[f(Y)]}{\sqrt{\EE_{Y \sim \cQ}[f(Y)^2]}} > \val(\phi).
\end{equation}
\end{proposition}

\begin{remark}[Runtime]
\label{rem:runtime}
The runtime bound $\cT$ may stand for a class of asymptotic runtimes such as $O(n)$ or $\poly(n)$, or it may be $\infty$ (no bound on runtime). The runtime for the positive result need not be uniform in $\alpha,\beta$, e.g., if $\cT = \poly(n)$ it would be fine to have a different algorithm for each $\alpha,\beta$ with runtime $O(n^{C(\alpha,\beta)})$.

The precise meaning of the runtime $O(\cT)$ for $f$ is as follows. A finite list of $(\alpha,\beta)$ pairs, depending only on $\phi,\alpha^*,\beta^*$ (not $n$), is chosen in advance and hard-coded into the algorithm. The algorithm runs the corresponding tests that asymptotically achieve these points, to get a vector of responses $s \in \{\fp,\fq\}^r$ where $r = r(\phi,\alpha^*,\beta^*)$. Finally, a hard-coded lookup table maps each of the $2^r = O(1)$ possible values for $s$ to a corresponding output $f(Y) > 0$.
\end{remark}

\noindent We now state our conclusion specialized to the case of spiked Wigner: $\phi_\lambda$ is the best possible ROC curve for computationally efficient algorithms.

\begin{corollary}\label{cor:main-comp}
Consider the spiked Wigner testing problem and assume Conjecture~\ref{conj:new}. Fix any $\epsilon > 0$, any $\lambda \in (0,1)$, and any $\pi$ satisfying Assumption~\ref{assum:pi}. Any $(\alpha,\beta) \in [0,1]^2$ with $\beta > \phi_\lambda(\alpha)$ is \emph{not} asymptotically achievable in time $\exp(O(n^{1-\epsilon}))$.
\end{corollary}

\noindent In general, the runtime bound $\exp(O(n^{1-\epsilon}))$ cannot be made any larger: for some $\lambda,\pi$ pairs, strong detection is possible in time $\exp(O(n))$ by directly evaluating the likelihood ratio~\eqref{eq:wig-lr}.

\begin{proof}
We will apply Proposition~\ref{prop:reduction} with $\cP = \PP_\lambda$, $\cQ = \PP_0$, $\phi = \phi_\lambda$ (which satisfies Assumption~\ref{assum:phi} by Lemma~\ref{lem:check-phi-assum}), and $\cT = \exp(O(n^{1-\epsilon}))$. Theorem~\ref{thm:positive} gives the required positive result. Assume on the contrary that some $(\alpha^*,\beta^*)$ with $\beta^* > \phi_\lambda(\alpha^*)$ is asymptotically achievable in time $\exp(O(n^{1-\epsilon}))$. By Proposition~\ref{prop:reduction} we obtain $f$ computable in time $\exp(O(n^{1-\epsilon}))$ such that
\begin{equation}\label{eq:contr-1}
\liminf_{n \to \infty} R_\lambda(f) > \val(\phi_\lambda) = (1-\lambda^2)^{-1/4},
\end{equation}
where the final equality is Lemma~\ref{lem:calc-val}. Furthermore, $f$ outputs values from a finite list of positive constants, meaning $f(Y)$ always lies in the range $[1/B,B]$ for a constant $B \ge 1$; see Remark~\ref{rem:runtime}. On the other hand, Conjecture~\ref{conj:new} implies that $f$ must satisfy
\begin{equation}\label{eq:contr-2}
\limsup_{n \to \infty} R_\lambda(f) \le \limsup_{n \to \infty} \sup_{g \,:\, \deg(g) \le n^{1-\epsilon/2}} R_\lambda(g) = (1-\lambda^2)^{-1/4},
\end{equation}
where the final equality is Theorem~\ref{thm:L-bound} combined with~\eqref{eq:op-def}. Together,~\eqref{eq:contr-1} and~\eqref{eq:contr-2} give a contradiction.
\end{proof}

\subsection{Discussion}

In this work we have presented a general-purpose framework for arguing about the precise testing error achievable by computationally efficient algorithms --- the first such framework, to our knowledge. Our method involves a two-stage argument: we first conjecture that low-degree polynomials are optimal for the related task of optimizing the ratio $R_\lambda(f)$, and we then provably characterize the best possible ROC curve for efficient algorithms conditional on this conjecture.

An alternative approach, and a potential direction for future work, would be to prove unconditionally that a particular class of tests (say, those based on thresholding low-degree polynomials) cannot surpass a particular ROC curve. At this point, even the simpler task of showing that polynomial threshold functions fail to achieve \emph{strong detection} is not fully understood, as the existing results along these lines have technical conditions that are likely unnecessary (see Section~4.1 of~\cite{ld-notes} and Section~2.3 of~\cite{sparse-clustering}).

While we view unconditional results as an important future direction, we also emphasize the merits of our current approach. Notably, the only low-degree calculation required is to bound $\|L^{\le D}\|$, which tends to be relatively simple. This has allowed us to prove a very sharp result that captures the precise threshold at $\lambda = 1$, the precise asymptotic ROC curve $\phi_\lambda$, and the nearly optimal degree $D = o(n/\log n)$. In contrast, other low-degree arguments not based on $\|L^{\le D}\|$ --- such as~\cite{SW-estimation,MW-amp}, which deal with \emph{estimation} rather than testing --- tend to be more difficult, and the known results are less sharp. For this reason, the approach we have presented in this work appears to be an especially user-friendly tool that may be useful for other problems.

We now discuss some specific extensions of our result that appear to be feasible directions for future work. First, we have considered the spiked Wigner model with GOE noise, but one can imagine relaxing the assumptions on the noise matrix $W$. For instance, one non-GOE (but still Gaussian) noise model takes the diagonal entries of $W$ to be $\cN(0,\sigma^2/n)$ instead of $\cN(0,2/n)$, while the off-diagonal entries remain $\cN(0,1/n)$. The analysis of LSS has been extended to this case~\cite{weak-wigner} (with a different ROC curve that depends on $\sigma$), and we have shown in Remark~\ref{rem:non-GOE} how our method can prove \emph{statistical} optimality of LSS when the spike prior is Rademacher. We expect it should also be possible to prove \emph{computational} optimality of LSS for more general priors, generalizing our Corollary~\ref{cor:main-comp}. This would require a probabilistic analysis of the additional second term discussed in Remark~\ref{rem:non-GOE}, which is no longer deterministic for general priors.

A more significant generalization of the Wigner model would allow a known non-Gaussian distribution for the entries of $W$. While LSS has also been analyzed here, it is not the optimal test: a better test involves first applying an entrywise transformation followed by LSS~\cite{weak-wigner}. We expect this modified test to be optimal among computationally efficient tests for a broad class of spike priors, and this can potentially be approached using our method. This would require a precise analysis of the low-degree likelihood ratio in this more general setting.

Throughout, we have assumed the spike prior $\pi$ is subgaussian, and it is plausible that this could be relaxed. Doing so may require a more complicated argument such as a conditional second moment computation. Certainly it is necessary to make some assumption on the growth of $\pi$'s moments or else LSS may cease to be the optimal test: if $\pi$ is so heavy-tailed that the maximum of $n$ independent copies exceeds $\tilde{O}(n^{1/4})$ then the maximum entry of the observed matrix $Y$ will achieve strong detection for all $\lambda > 0$. We also note that our assumption that $\pi$ have mean 0 is needed, or else LSS will be sub-optimal compared to a simple test based on the sum of all entries of $Y$.

We also hope that our approach may become useful beyond the spiked Wigner setting. A key ingredient in our approach is a precise bound on the norm of the low-degree likelihood ratio. The main step in deriving this precise limit is to show boundedness of the norm, as this allows us to deduce the desired convergence in expectation from the corresponding convergence in distribution. There are many settings where this boundedness has already been established, including spiked Wishart matrices~\cite{sk-cert} and the stochastic block model~\cite{HS-bayesian,spectral-planting}, just to name a few. However, the precise positive results for testing do not yet seem to be known in these settings.

\subsection{Proof Overview}
\label{sec:pf-overview}

The bulk of our technical work goes into proving Proposition~\ref{prop:reduction}. A key fact underlying the proof is the following: given an achievable ROC curve $\phi$, one can construct a function $f$ with $\EE_{\cP}[f]/\sqrt{\EE_{\cQ}[f^2]} = \val(\phi)$. We will give an overview of how this is done and where the formula for $\val(\cdot)$ comes from.

For intuition it will help to consider a simplified setting: imagine our ROC curve $\phi$ is achieved by thresholding a particular test statistic that only takes $r$ different values. There are $r+1$ choices for the threshold, yielding $r+1$ different tests. Say test $i$ achieves $(\alpha,\beta) = (a_i,b_i)$ for $i = 0,1,\ldots,r$, where $0 = a_0 < a_1 < \cdots < a_r = 1$ and $0 = b_0 < b_1 < \cdots < b_r = 1$. Also assume that no $(a_i,b_i)$ is dominated by the convex hull of other $(a_j,b_j)$ pairs, or else test $i$ is redundant and can be removed. This means $\phi$ is the concave piecewise linear function formed by connecting each $(a_i,b_i)$ to $(a_{i+1},b_{i+1})$ with a line segment. (These points are achieved by taking appropriate probabilistic combinations between tests $i$ and $i+1$.) Since we're assuming all our tests are thresholding the same underlying statistic, there will be, for any input $Y$, a unique $i^* = i^*(Y) \in \{0,1,\ldots,r-1\}$ such that all tests $i \le i^*$ output $\fq$ and all tests $i > i^*$ output $\fp$. By definition of $a_i,b_i$, the probability that $i^* = i$ is $a_{i+1} - a_i$ under $\cQ$ and $b_{i+1} - b_i$ under $\cP$. Letting $Q^*$ and $P^*$ denote the distributions of $i^*$ under $\cQ$ and $\cP$ respectively, it is a standard fact that the function $f$ maximizing $\EE_{\cP}[f]/\sqrt{\EE_{\cQ}[f^2]}$ (given access to $i^*$) is the likelihood ratio (see e.g.,~\cite[Proposition~1.9]{ld-notes})
\[ f(Y) = \frac{dP^*}{dQ^*}(i^*) = \frac{b_{i^*+1} - b_{i^*}}{a_{i^*+1} - a_{i^*}}, \]
which is the slope of the line connecting $(a_{i^*},b_{i^*})$ and $(a_{i^*+1},b_{i^*+1})$. For this choice of $f$, both $\EE_{\cP}[f]$ and $\EE_{\cQ}[f^2]$ are equal to
\[ \sum_{i=0}^{r-1} \frac{(b_{i+1}-b_i)^2}{a_{i+1}-a_i} = \sum_{i=0}^{r-1} \left(\frac{b_{i+1}-b_i}{a_{i+1}-a_i}\right)^2 (a_{i+1}-a_i) = \val(\phi)^2, \]
and so $\EE_{\cP}[f]/\sqrt{\EE_{\cQ}[f^2]} = \val(\phi)$ as desired.

The full proof of Proposition~\ref{prop:reduction} is more involved for a number of reasons. We do not assume the ROC curve is piecewise linear, but the proof will involve an appropriate discretization. We also do not assume the ROC curve is achieved by thresholding a single statistic at varying thresholds: there may be no relation between the tests that achieve two different $(\alpha,\beta)$ pairs. Finally, we must show that given a hypothetical achievable point $(\alpha^*,\beta^*)$ \emph{above} the graph of $\phi$, we can achieve an improved ROC curve whose value is \emph{strictly} larger than $\val(\phi)$ (see Figure~\ref{fig:pf} for illustration). We note that the value $\val(\phi)$ is only meaningful for concave functions $\phi$.

\section{Proofs}
\label{sec:proofs}

\subsection{Proof of Theorem~\ref{thm:L-bound}}

A standard formula for $\|(L_\lambda)^{\le D}\|$ gives (see Theorem~2.6 of~\cite{ld-notes})
\begin{equation}\label{eq:L-wigner}
\|(L_\lambda)^{\le D}\|^2 = \E_{x,x'} \exp^{\le D}\left(\frac{\lambda^2 n \langle x,x' \rangle^2}{2 \|x\|^2 \|x'\|^2}\right)
\end{equation}
where $x'$ is an independent copy of $x$ (with entries i.i.d.\ from $\pi$), and $\exp^{\le D}(z) := \sum_{d=0}^D z^d/d!$ (the degree-$D$ Taylor polynomial for $\exp(z)$). Define the random variable $A = A_n \ge 0$ by
\begin{equation}\label{eq:def-A}
A := \frac{\lambda^2 n \langle x,x' \rangle^2}{2 \|x\|^2 \|x'\|^2},
\end{equation}
so that
\begin{equation}\label{eq:L-decomp}
\|(L_\lambda)^{\le D}\|^2 = \EE \exp^{\le D}(A) = \EE[\One_{A \le t} \cdot \exp^{\le D}(A)] + \EE[\One_{A > t} \cdot \exp^{\le D}(A)]
\end{equation}
where $t = t_n \ge 0$ is a threshold to be chosen later. The first term on the right-hand side is bounded by
\begin{equation}\label{eq:L-decomp-2}
\EE[\One_{A \le t} \cdot \exp(A)] - \sum_{d=D+1}^\infty \frac{t^d}{d!} \le \EE[\One_{A \le t} \cdot \exp^{\le D}(A)] \le \EE[\One_{A \le t} \cdot \exp(A)],
\end{equation}
where the second inequality relies on $A \ge 0$ so that every term in the Taylor series for exp is non-negative. We will proceed to bound the various terms in~\eqref{eq:L-decomp} and~\eqref{eq:L-decomp-2} separately.

\begin{lemma}\label{lem:L-term-1}
If $t = o(D)$ then
\[ \sum_{d=D+1}^\infty \frac{t^d}{d!} = o(1). \]
\end{lemma}
\begin{proof}
Using the standard bound $d! \ge (d/e)^d$,
\[ \sum_{d=D+1}^\infty \frac{t^d}{d!} \le \sum_{d=D+1}^\infty \left(\frac{et}{d}\right)^d \le \sum_{d=D+1}^\infty \left(\frac{et}{D}\right)^d \le \sum_{d=1}^\infty \left(\frac{et}{D}\right)^d = o(1), \]
since $t=o(D)$.
\end{proof}

\begin{lemma}\label{lem:L-term-2}
If $t=\omega(1)$ and $t \le \rho n$ for a particular constant $\rho = \rho(\lambda,\pi) > 0$, then
\[ \lim_{n \to \infty} \EE[\One_{A \le t} \cdot \exp(A)] = (1-\lambda^2)^{-1/2}. \]
\end{lemma}
\begin{proof}
By the central limit theorem, we have convergence in distribution $\langle x,x' \rangle/\sqrt{n} \stackrel{d}{\to} \cN(0,1)$. Also, $\|x\|^2/n \stackrel{d}{\to} 1$ and the same for $x'$. As a result, $A \stackrel{d}{\to} \lambda^2 \chi_1^2/2$ where $\chi_1^2$ is a chi-squared random variable with 1 degree of freedom. Since $t = \omega(1)$, we have $\One_{A \le t} \cdot \exp(A) \stackrel{d}{\to} \exp(\lambda^2 \chi_1^2/2)$. We aim to conclude the corresponding convergence of the expectation $\EE[\One_{A \le t} \cdot \exp(A)] \to \EE[\exp(\lambda^2 \chi_1^2/2)]$. This completes the proof because, using the chi-squared moment-generating function, $\EE[\exp(\lambda^2 \chi_1^2/2)] = (1-\lambda^2)^{-1/2}$.

Convergence in distribution plus uniform integrability implies convergence of the expectation, so it remains to verify that the sequence $X_n := \One_{A_n \le t_n} \cdot \exp(A_n)$ is uniformly integrable, i.e., for every $\epsilon > 0$ there exists $K \ge 0$ such that for all $n$ we have $\EE[\One_{|X_n| \ge K} \cdot |X_n|] \le \epsilon$. Lemma~\ref{lem:unif-int} below establishes for all $n$ that $\EE[|X_n|^{1+\gamma}] \le C$ for some positive constants $\gamma,C$ (depending on $\lambda,\pi$). This implies uniform integrability because
\[ \EE[\One_{|X_n| \ge K} \cdot |X_n|] = \EE[\One_{|X_n| \ge K} \cdot |X_n|^{1+\gamma} |X_n|^{-\gamma}] \le \EE[\One_{|X_n| \ge K} \cdot |X_n|^{1+\gamma} K^{-\gamma}] \le CK^{-\gamma}, \]
which can be made smaller than $\epsilon$ by choosing $K$ sufficiently large.
\end{proof}

We note that the rest of the argument is similar to~\cite[Theorem~3.9]{ld-notes}, with the main difference being that we have normalized the spike (a convention we have adopted for consistency with~\cite{weak-wigner}). We will defer many of the details to Appendix~\ref{app:analysis-ldlr} and state only the main claims here. The following lemma, proved in Appendix~\ref{app:pf-lem-unif-int}, was used to establish uniform integrability above.

\begin{lemma}\label{lem:unif-int}
Fix $\lambda \in (0,1)$ and choose $\gamma > 0$ small enough so that $(1+\gamma)\lambda^2 < 1$. If $t \le \rho n$ for a particular constant $\rho = \rho(\lambda,\pi,\gamma) > 0$ then
\[ \EE[\One_{A \le t} \cdot \exp((1+\gamma)A)] = O(1). \]
\end{lemma}

\noindent We now return to the final term in~\eqref{eq:L-decomp}.

\begin{lemma}\label{lem:L-term-3}
For any sequences $D = D_n$ and $t = t_n$ that scale as $D = o(n/\log n)$ and $t = \Omega(n)$,
\[ \EE[\One_{A > t} \cdot \exp^{\le D}(A)] = o(1). \]
\end{lemma}

\noindent The proof of Lemma~\ref{lem:L-term-3} is deferred to Appendix~\ref{app:pf-lem-L-term-3}. The proof of Theorem~\ref{thm:L-bound} now follows by combining the ingredients above; the full details of this are deferred to Appendix~\ref{app:pf-thm-L-bound}.

\subsection{Proof of Proposition~\ref{prop:reduction}}

We will first argue that $\alpha^* \in (0,1)$ without loss of generality. First, since $\phi(1) = 1$ and $\phi(\alpha^*) < \beta^* \le 1$, we must have $\alpha^* \ne 1$. Second, if $\alpha^* = 0$, we will find a \emph{different} asymptotically achievable point $(\alpha,\beta)$ with $\beta > \phi(\alpha)$ and $\alpha > 0$, to use in place of $(\alpha^*,\beta^*)$. Since $\phi$ is continuous, this can be done via the following test: with probability $p$ for a sufficiently small constant $p > 0$, output $\fp$; otherwise, apply the original decision rule that achieved $(\alpha^*,\beta^*)$. In the sequel we therefore assume $\alpha^* > 0$.

Define $\psi: [0,1] \to [0,1]$ to be the upper concave envelope of $\phi$ and $(\alpha^*,\beta^*)$, as illustrated in Figure~\ref{fig:pf}. By this we mean the minimal concave curve such that $\psi(\alpha) \ge \phi(\alpha)$ for all $\alpha \in [0,1]$ and $\psi(\alpha^*) = \beta^*$. Note that $\psi$ has the following structure for certain constants $0 < A_1 < \alpha^* < A_2 < 1$.
\begin{itemize}
\item For $\alpha \in [0,A_1] \cup [A_2,1]$, $\psi(\alpha) = \phi(\alpha)$.
\item For $\alpha \in [A_1,\alpha^*]$, the graph of $\psi$ is a straight line connecting $(A_1,\phi(A_1))$ and $(\alpha^*,\beta^*)$, and this line has slope $\phi'(A_1)$.
\item For $\alpha \in [\alpha^*,A_2]$, the graph of $\psi$ is a straight line connecting $(\alpha^*,\beta^*)$ and $(A_2,\phi(A_2))$, and this line has slope $\phi'(A_2)$.
\item Finally, $\phi'(A_1) > \phi'(A_2)$, since $\phi$ is concave and $(\alpha^*,\beta^*)$ lies strictly above the graph of $\phi$.
\end{itemize}
\noindent Here we have used some basic properties of $\phi$ from Assumption~\ref{assum:phi}, namely the fact that $\phi'$ is continuous and decreasing, as well as the limits of $\phi'$ near $0$ and $1$.

\begin{lemma}\label{lem:pushout}
$\val(\psi)$ exists and $\val(\psi) > \val(\phi)$.
\end{lemma}

\begin{proof}
The existence of $\val(\psi)$ is immediate from the existence of $\val(\phi)$ and the piecewise structure of $\psi$ discussed above. Since $\phi$ and $\psi$ coincide outside the interval $[A_1,A_2]$, it suffices to show
\begin{equation}\label{eq:pushout-goal}
\int_{A_1}^{A_2} (\psi'(\alpha))^2 \,d\alpha > \int_{A_1}^{A_2} (\phi'(\alpha))^2 \,d\alpha.
\end{equation}
Introduce the shorthand $\Delta_1 := \alpha^* - A_1$ and $\Delta_2 := A_2 - \alpha^*$. Since $\psi$ is piecewise linear on $[A_1,A_2]$, the left-hand side can be computed directly:
\begin{equation}\label{eq:pushout-lhs}
\int_{A_1}^{A_2} (\psi'(\alpha))^2 \,d\alpha = (\phi'(A_1))^2 \Delta_1 + (\phi'(A_2))^2 \Delta_2.
\end{equation}
Since $\phi(A_1) = \psi(A_1)$, $\phi(\alpha^*) < \psi(\alpha^*)$, and $\phi(A_2) = \psi(A_2)$, we have for some $\eta > 0$,
\begin{equation}\label{eq:int-eta-1}
\int_{A_1}^{\alpha^*} \phi'(\alpha) \,d\alpha = \int_{A_1}^{\alpha^*} \psi'(\alpha) \,d\alpha - \eta = \phi'(A_1) \Delta_1 - \eta
\end{equation}
and
\begin{equation}\label{eq:int-eta-2}
\int_{\alpha^*}^{A_2} \phi'(\alpha) \,d\alpha = \int_{\alpha^*}^{A_2} \psi'(\alpha) \,d\alpha + \eta = \phi'(A_2) \Delta_2 + \eta.
\end{equation}
Also note that an integrable function $f$ on $[a,b]$ with $0 \le m \le f(x) \le M$ for all $x \in [a,b]$ must satisfy
\begin{align}
\int_a^b f(x)^2 \,dx &= \int_a^b (f(x) - m)(f(x) + m) \,dx + m^2 (b-a) \nonumber \\
&\le (M+m) \int_a^b (f(x)-m)\,dx + m^2 (b-a) \nonumber \\
&= (M+m) \int_a^b f(x)\,dx - Mm(b-a).
\label{eq:mfM-fact}
\end{align}
Recalling that $\phi'$ is positive and decreasing, combine~\eqref{eq:int-eta-1},\eqref{eq:int-eta-2},\eqref{eq:mfM-fact} to bound the right-hand side of~\eqref{eq:pushout-goal}:
\begin{align*}
\int_{A_1}^{A_2} (\phi'(\alpha))^2 \,d\alpha &= \int_{A_1}^{\alpha^*} (\phi'(\alpha))^2 \,d\alpha + \int_{\alpha^*}^{A_2} (\phi'(\alpha))^2 \,d\alpha \\
&\le [\phi'(A_1) + \phi'(\alpha^*)] [\phi'(A_1) \Delta_1 - \eta] - \phi'(A_1) \phi'(\alpha^*) \Delta_1 \\
&\qquad + [\phi'(\alpha^*) + \phi'(A_2)] [\phi'(A_2) \Delta_2 + \eta] - \phi'(\alpha^*) \phi'(A_2) \Delta_2 \\
&= (\phi'(A_1))^2 \Delta_1 + (\phi'(A_2))^2 \Delta_2 - \eta[\phi'(A_1) - \phi'(A_2)].
\end{align*}
Compare this with~\eqref{eq:pushout-lhs} and recall $\phi'(A_1) > \phi'(A_2)$ to complete the proof.
\end{proof}

Our next step will be to approximate $\psi$ using a finite collection of points $u = (u_0,\ldots,u_r)$ where $u_i = (a_i,\psi(a_i))$ for some choice of $0 = a_0 < a_1 < \cdots < a_r = 1$. By virtue of lying on the graph of $\psi$, these points will be in ``concave position,'' which we define as follows.
\begin{definition}[Concave position and $\conc$]
Let $u = (u_0,\ldots,u_r)$ be a sequence of points in $[0,1]^2$, sorted by strictly ascending first coordinate, with $u_0 = (0,0)$ and $u_r = (1,1)$. We say $u$ is in \emph{concave position} if the slope of the line through $u_i$ and $u_{i+1}$ is strictly positive and strictly decreasing as a function of $i$. In this case, we let $\conc(u)$ denote the upper convex envelope of these points, i.e., the piecewise linear function $[0,1] \to [0,1]$ whose graph on $[a_i,a_{i+1}]$ is a straight line connecting $u_i$ and $u_{i+1}$.
\end{definition}

\begin{lemma}\label{lem:construct-u}
There exist values $r \ge 1$ and $0 = a_0 < a_1 < \cdots < a_r = 1$ such that the points $u = (u_0,\ldots,u_r)$ where $u_i = (a_i,\psi(a_i))$ are in concave position with
\[ \val(\conc(u)) > \val(\phi). \]
\end{lemma}

\begin{proof}
In light of Lemma~\ref{lem:pushout} we can write $\val(\psi)^2 = \val(\phi)^2 + \epsilon$ for some $\epsilon > 0$. Since $\val(\psi)$ exists, it is possible to choose $a_1,a_{r-1}$ with $0 < a_1 < A_1$ and $A_2 < a_{r-1} < 1$ so that $\int_0^{a_1} (\psi'(\alpha))^2\,d\alpha \le \epsilon/6$ and $\int_{a_{r-1}}^1 (\psi'(\alpha))^2\,d\alpha \le \epsilon/6$. We will also include $A_1,\alpha^*,A_2$ in the list of $a_i$'s. It remains to partition the intervals $[a_1,A_1]$ and $[A_2,a_{r-1}]$. In each case, choose a fine enough partition so that each sub-interval $[a_i,a_{i+1}]$ satisfies $\psi'(a_i) \le (1+\gamma) \psi'(a_{i+1})$ for a constant $\gamma > 0$ to be chosen later. This is possible because $\phi'$ is continuous and strictly positive, and $\psi' = \phi'$ on $[a_1,A_1] \cup [A_2,a_{r-1}]$. Let $I$ denote the set of $i$ that index the sub-intervals $[a_i,a_{i+1}]$ of $[a_1,A_1] \cup [A_2,a_{r-1}]$. For each $i \in I$, the mean value theorem implies that the slope $m_i$ of the line connecting $u_i$ and $u_{i+1}$ satisfies
\[ \psi'(a_{i+1}) \le m_i \le \psi'(a_i) \le (1+\gamma)\psi'(a_{i+1}). \]
Note that the interval $[A_1,A_2]$ has the same contribution to both $\val(\conc(u))$ and $\val(\psi)$, and so
\begin{align*}
\val(\psi)^2 - & \val(\conc(u))^2 \le \sum_{i \in I} \left[\int_{a_i}^{a_{i+1}} (\psi'(\alpha))^2\,d\alpha - m_i^2(a_{i+1}-a_i)\right] \\
&\qquad\qquad\qquad\qquad + \int_{[0,a_1] \cup [a_{r-1},1]} (\psi'(\alpha))^2\,d\alpha \\
&\le \sum_{i \in I} \left[(\psi'(a_i))^2 (a_{i+1}-a_i) - m_i^2(a_{i+1}-a_i)\right] + \epsilon/3 \\
&\le \sum_{i \in I} \left[(1+\gamma)^2 m_i^2 (a_{i+1}-a_i) - m_i^2(a_{i+1}-a_i)\right] + \epsilon/3 \\
&= (2\gamma + \gamma^2) \sum_{i \in I} m_i^2(a_{i+1}-a_i) + \epsilon/3 \;\le\; (2\gamma + \gamma^2) (\psi'(a_1))^2 + \epsilon/3,
\end{align*}
implying $\val(\psi)^2 - \val(\conc(u))^2 \le 2\epsilon/3$ for a sufficiently small choice of $\gamma > 0$. Recalling $\val(\psi)^2 = \val(\phi)^2 + \epsilon$, this completes the proof that $\val(\conc(u)) > \val(\phi)$. Finally, remove any redundant $u_i$'s for which $m_{i-1} = m_i$, noting that this has no effect on $\val(\conc(u))$. The remaining points are in concave position because they lie on the graph of $\psi$, which is concave and strictly increasing.
\end{proof}

Next we modify the points $u$ slightly to create a new sequence of points $v = (v_0,\ldots,v_r)$ where $v_i = (a_i,b_i)$. The $\alpha$-coordinates are inherited from the $u$'s and the $\beta$-coordinates will be decreased slightly to ensure these points are asymptotically achievable. This step is only necessarily because we have only assumed asymptotic achievability for points strictly below the graph of $\phi$, not on it. The proof is deferred to Appendix~\ref{app:pf-lem-construct-v}.

\begin{lemma}\label{lem:construct-v}
Let $r$ and $\{a_i\}$ be as in Lemma~\ref{lem:construct-u}. There exist values $0 = b_0 < b_1 < \cdots < b_r = 1$ such that the points $v = (v_0,\ldots,v_r)$ where $v_i = (a_i,b_i)$ are in concave position with
\[ \val(\conc(v)) > \val(\phi), \]
and furthermore, $b_i < \psi(a_i)$ for each $i \notin \{0,r\}$.
\end{lemma}

We now have a finite number of points $v_0,\ldots,v_r$ such that each is asymptotically achievable in time $\cT$. These points depend only on $\phi,\alpha^*,\beta^*$ but not $n$. To complete the proof, we construct an algorithm to compute a function $f$ achieving~\eqref{eq:ratio-goal}. The points $v_0,\ldots,v_r$, as well as the algorithms for asymptotically achieving each of them, will be hard-coded into our algorithm for $f$.

\begin{proof}[Proof of Proposition~\ref{prop:reduction}]
Given input $Y$, we describe the algorithm to compute $f(Y)$. Let $t_0,\ldots,t_r$ denote the tests that asymptotically achieve the points $v_0,\ldots,v_r$ from Lemma~\ref{lem:construct-v}, and let $s = (s_0,\ldots,s_r)$ denote the sequence of outputs $s_i = t_i(Y)$. Since $v_0 = (0,0)$ and $v_r = (1,1)$, we always have $s_0 = \fq$ and $s_r = \fp$. Our goal is to choose $f(Y)$ so that
\begin{equation}\label{eq:ratio-goal-2}
\frac{\EE_{Y \sim \cP}[f(Y)]}{\sqrt{\EE_{Y \sim \cQ}[f(Y)^2]}} \ge \val(\conc(v)) - o(1),
\end{equation}
which implies~\eqref{eq:ratio-goal} since $\val(\conc(v)) > \val(\phi)$ by Lemma~\ref{lem:construct-v}.

For intuition, it is natural to expect the rejection regions of the $t_i$'s to be nested in the sense that $t_{i+1}(Y) = \fp$ whenever $t_i(Y) = \fp$. If this is the case then the sequence $s$ is always equal to one of the ``monotone'' sequences
\[ S^{(0)} = (\fq,\fp,\fp,\ldots,\fp), \quad S^{(1)} = (\fq,\fq,\fp,\ldots,\fp), \quad \ldots, \quad S^{(r-1)} = (\fq,\fq,\ldots,\fq,\fp). \]
If $s = S^{(j)}$ for some $j$, let $f(Y) = \ell_j$, which, recall, is the slope of the line connecting $v_i$ and $v_{i+1}$. In the case where the rejection regions of the $t_i$'s are nested, this rule achieves~\eqref{eq:ratio-goal-2}. However, there is no guarantee that the rejection regions will be nested, and so we need to also decide on $f(Y)$ for non-monotone $s$.

The general construction for $f(Y)$ is as follows. For $i = 0,\ldots,r-1$ define
\[ \sigma_i(s) = \begin{cases} +1 & \text{if } s_i = \fq,\, s_{i+1} = \fp, \\ -1 & \text{if } s_i = \fp,\, s_{i+1} = \fq, \\ 0 & \text{otherwise.} \end{cases} \]
Then choose any $f(Y) > 0$ satisfying
\[ \sum_{i=0}^{r-1} \sigma_i(s) \ell_i \le f(Y) \le \sqrt{\sum_{i=0}^{r-1} \sigma_i(s) \ell_i^2}, \]
for example, $f(Y)$ can be defined to simply equal the left-hand side above. The interval for possible $f(Y)$ values is nonempty due to Lemma~\ref{lem:alt-sum} below, using the following facts: $\ell_i$ is positive and decreasing in $i$, by definition of concave position; and the nonzero $\sigma_i$'s alternate in sign, starting with positive. Since there are only a finite number of possible values for $s$, a lookup table with the corresponding $f(Y)$ values can be hard-coded into the algorithm. Note that if $s$ is a monotone sequence $S^{(j)}$, we recover the rule $f(Y) = \ell_j$ from above.

Since $f(Y)$ depends only on $s$, we write $f(s) = f(Y)$. Now compute
\begin{align*}
\EE_\cP[f] &= \sum_s \cP(s) f(s) \ge \sum_s \cP(s) \sum_{i=0}^{r-1} \sigma_i(s) \ell_i = \sum_{i=0}^{r-1} \ell_i \sum_s \sigma_i(s) \cP(s) \\
&= \sum_{i=0}^{r-1} \ell_i \, [\cP(s_{i+1} = \fp) - \cP(s_i = \fp)] \\
&= \sum_{i=1}^r \cP(s_i = \fp) (\ell_{i-1} - \ell_i) \qquad\qquad \text{where } \ell_r := 0 \text{ and using } \cP(s_0 = \fp) = 0 \\
&\ge \sum_{i=1}^r b_i (\ell_{i-1} - \ell_i) - o(1) = \sum_{i=0}^{r-1} \ell_i (b_{i+1} - b_i) - o(1) = \val(\conc(v))^2 - o(1).
\end{align*}
Similarly,
\begin{align*}
\EE_\cQ[f^2] &= \sum_s \cQ(s) f(s)^2 \le \sum_s \cQ(s) \sum_{i=0}^{r-1} \sigma_i(s) \ell_i^2 = \sum_{i=0}^{r-1} \ell_i^2 \sum_s \sigma_i(s) \cQ(s) \\
&= \sum_{i=0}^{r-1} \ell_i^2 \, [\cQ(s_{i+1} = \fp) - \cQ(s_i = \fp)] = \sum_{i=1}^r \cQ(s_i = \fp) (\ell_{i-1}^2 - \ell_i^2) \\
&\le \sum_{i=1}^r a_i (\ell_{i-1}^2 - \ell_i^2) - o(1) = \sum_{i=0}^{r-1} \ell_i^2 (a_{i+1} - a_i) - o(1) = \val(\conc(v))^2 - o(1).
\end{align*}
Together, the above calculations imply~\eqref{eq:ratio-goal-2} as desired.
\end{proof}

\noindent Finally, the following fact was used above. Its proof is deferred to Appendix~\ref{app:pf-lem-alt-sum}.

\begin{lemma}\label{lem:alt-sum}
For an integer $m \ge 1$, suppose $h_0 > h_1 > \cdots > h_{m-1} > 0$. Then
\[ 0 < \sum_{i=0}^{m-1} (-1)^i h_i \le \sqrt{\sum_{i=0}^{m-1} (-1)^i h_i^2}. \]
\end{lemma}

\appendix

\section{LSS and its ROC Curve}

\subsection{Proof of Theorem~\ref{thm:positive}}
\label{app:pf-thm-positive}

\begin{proof}[Proof of Theorem~\ref{thm:positive}]
Our null distribution is identical to that of~\cite{weak-wigner} (taking their parameters to have values $w_2 = 2$, $w_4 = 3$), and our alternative is identical to theirs (with $\omega = \lambda^2$) except on the event $\|x\| = 0$. This event occurs with probability $c^n$ for a constant $c \in [0,1)$, namely $c := \Pr(\pi = 0)$. Thus, if we run the algoritm of~\cite{weak-wigner} on our model, $\alpha$ will be unchanged and $\beta$ will only change by an additive $o(1)$.

For convenience, we repeat the definition of the LSS test statistic:
\begin{equation}\label{eq:lss-2}
\sum_{i=1}^n h_\lambda(\mu_i) \qquad \text{with} \qquad h_\lambda(\mu) := -\log(1 - \lambda \mu + \lambda^2),
\end{equation}
where $\mu_1 \ge \mu_2 \ge \cdots \ge \mu_n$ are the eigenvalues of the input matrix $Y$. The test statistic $L_\omega$ considered in~\cite{weak-wigner} is a centered version of~\eqref{eq:lss-2}, namely $\sum_{i=1}^n h_\lambda(\mu_i) - Cn$ for a constant $C = C(\lambda)$; it is shown in~\cite{weak-wigner} that $L_\omega$ converges in distribution (as $n \to \infty$) to
\[ \cN\left(-\frac{1}{2} \log(1-\lambda^2) \, , \; -2\log(1-\lambda^2)\right) \qquad \text{under } \PP_0, \]
\[ \cN\left(-\frac{3}{2} \log(1-\lambda^2) \, , \; -2\log(1-\lambda^2)\right) \qquad \text{under } \PP_\lambda. \]
Let $\mu = -\log(1-\lambda^2) > 0$ so that the statistic $(L_\omega - \mu)/\sqrt{2\mu}$ converges in distribution to $\cN(\pm \sqrt{\mu/8}, 1)$ where the plus sign holds under $\PP_\lambda$ and the minus sign holds under $\PP_0$. By thresholding this statistic at some constant $\tau \in \RR$, we obtain a test whose size and power converge to
\begin{align*}
\alpha &= 1-\Phi(\tau+\sqrt{\mu/8}), \\
\beta &= 1-\Phi(\tau-\sqrt{\mu/8}).
\end{align*}
Solving for $\beta$ in terms of $\alpha$ yields
\begin{align*}
\beta &= 1 - \Phi\left[\left(\Phi^{-1}(1-\alpha) - \sqrt{\mu/8}\right) - \sqrt{\mu/8}\right] \\
&= 1 - \Phi\left[\Phi^{-1}(1-\alpha) - \sqrt{\mu/2}\right],
\end{align*}
or equivalently, $\beta = \phi_\lambda(\alpha)$.

It remains to show that the LSS test can be implemented in polynomial time (in the real RAM model; see Remark~\ref{rem:comp-model}). We will see that this can be done using a series of simple operations (such as matrix multiplication), without the need to compute individual eigenvalues. For any fixed $\lambda \in (0,1)$, we have with probability $1-o(1)$ that all the eigenvalues $\mu_i$ of the input matrix $Y$ lie in the interval $[-2-\epsilon,\,2+\epsilon]$ for any fixed $\epsilon > 0$ (see~\cite{FP-wigner,maida}); we will condition on this ``good'' event. In light of the prior discussion on convergence to Gaussian, to achieve the desired error rates we do not need to compute~\eqref{eq:lss-2} exactly but rather approximate it with additive error $o(1)$ whenever the above ``good'' event holds. To this end, it suffices to replace $h_\lambda$ by a polynomial approximation $\hat{h}_\lambda$ that uniformly approximates $h_\lambda$ with additive error $o(1/n)$ on the interval $[-2-\epsilon,\,2+\epsilon]$. The resulting statistic can be computed as a matrix polynomial in $Y$: $\sum_{i=1}^n \hat{h}_\lambda(\mu_i) = \Tr(\hat{h}_\lambda(Y))$.

To construct $\hat{h}_\lambda$, consider the degree-$d$ Taylor expansion of $h_\lambda$ centered at 0, namely
\[ \hat{h}_\lambda(x) := -\log(1+\lambda^2) + \sum_{k=1}^d \frac{1}{k} \left(\frac{\lambda}{1+\lambda^2}\right)^k x^k. \]
Since $h_\lambda$ is analytic on $\mathbb{C} \setminus [(1+\lambda^2)/\lambda,\infty)$, the Taylor series
\[ -\log(1+\lambda^2) + \sum_{k=1}^\infty \frac{1}{k} \left(\frac{\lambda}{1+\lambda^2}\right)^k x^k \]
converges to $h_\lambda(x)$ for all $|x| < (1+\lambda^2)/\lambda$. Since $\lambda \in (0,1)$ we have $\lambda/(1+\lambda^2) \in (0,1/2)$, and so we can choose $\epsilon > 0$ small enough so that $r := (2+\epsilon)\lambda/(1+\lambda^2) < 1$. This ensures that the Taylor series converges for all $x \in [-2-\epsilon,\,2+\epsilon]$. Now for $x \in [-2-\epsilon,\,2+\epsilon]$ we can write
\[ |h_\lambda(x) - \hat{h}_\lambda(x)| \le \sum_{k=d+1}^\infty \frac{1}{k} \left(\frac{\lambda}{1+\lambda^2}\right)^k |x|^k \le \sum_{k=d+1}^\infty \left(\frac{\lambda \cdot |x|}{1+\lambda^2}\right)^k \le \sum_{k=d+1}^\infty r^k = \frac{r^{d+1}}{1-r}. \]
To make this $o(1/n)$, it suffices to have $r^d/(1-r) \le 1/n^2$, which is achieved by choosing $d \ge [2\log n - \log(1-r)]/\log(1/r) = \Theta(\log n)$.
\end{proof}

\subsection{Proof of Lemma~\ref{lem:calc-val}}
\label{app:pf-lem-calc-val}

\begin{proof}[Proof of Lemma~\ref{lem:calc-val}]

From the proof of Theorem~\ref{thm:positive}, the curve $\beta = \phi_\lambda(\alpha)$ for $\alpha \in (0,1)$ admits a parametric representation
\begin{align}
\label{eq:parametric-a}
\alpha &= 1-\Phi(\tau+\sqrt{\mu/8}), \\
\label{eq:parametric-b}
\beta &= 1-\Phi(\tau-\sqrt{\mu/8}),
\end{align}
where $\mu = -\log(1-\lambda^2)$ and the variable $\tau$ ranges over $(-\infty,\infty)$. By elementary calculus,
\[ \Phi'(z) = \frac{1}{\sqrt{2\pi}} \exp(-z^2/2), \]
and so
\begin{align*}
\frac{d\alpha}{d\tau} &= -\frac{1}{\sqrt{2\pi}} \exp\left(-\frac{1}{2}(\tau + \sqrt{\mu/8})^2\right), \\
\frac{d\beta}{d\tau} &= -\frac{1}{\sqrt{2\pi}} \exp\left(-\frac{1}{2}(\tau - \sqrt{\mu/8})^2\right).
\end{align*}
Combining these,
\begin{equation}\label{eq:phi-prime}
\phi'_\lambda(\alpha) = \frac{d\beta}{d\tau} \cdot \frac{d\tau}{d\alpha} = \exp\left(\tau\sqrt{\mu/2}\right).
\end{equation}
Now
\begin{align*}
\val(\phi_\lambda)^2 &= \int_0^1 (\phi_\lambda'(\alpha))^2 \,d\alpha \\
&= \int_{\infty}^{-\infty} \exp\left(\tau \sqrt{2\mu}\right) \cdot \frac{-1}{\sqrt{2\pi}} \exp\left(-\frac{1}{2}(\tau + \sqrt{\mu/8})^2\right) \,d\tau \\
&= \frac{1}{\sqrt{2\pi}} \int_{-\infty}^\infty \exp\left(-\tau^2/2 + 3\tau\sqrt{2\mu}/4 - \mu/16\right) \,d\tau \\
&= \frac{1}{\sqrt{2\pi}} \int_{-\infty}^\infty \exp\left(-\frac{1}{2}(\tau - 3\sqrt{2\mu}/4)^2 + \mu/2\right) \,d\tau \\
&= \exp(\mu/2) \\
&= (1-\lambda^2)^{-1/2},
\end{align*}
and so $\val(\phi_\lambda) = (1-\lambda^2)^{-1/4}$.
\end{proof}

\subsection{Proof of Lemma~\ref{lem:check-phi-assum}}
\label{app:pf-lem-check-phi-assum}

\begin{proof}[Proof of Lemma~\ref{lem:check-phi-assum}]
Recall from~\eqref{eq:parametric-a},\eqref{eq:parametric-b} the parametric form $\alpha = \alpha(\tau)$, $\beta = \beta(\tau)$ for the curve $\beta = \phi_\lambda(\alpha)$. As $\tau$ increases from $-\infty$ to $\infty$, $\alpha$ and $\beta$ each decrease from $1$ to $0$. Also recall from~\eqref{eq:phi-prime} the equation for the derivative: $\phi'_\lambda(\alpha) = \exp(\tau\sqrt{\mu/2})$ where $\mu = -\log(1-\lambda^2) > 0$. Note that $\phi'_\lambda$ is continuous, strictly positive, and decreasing in $\alpha$ (since it's increasing in $\tau$). Furthermore, in the limit $\alpha \to 0^+$, i.e., $\tau \to \infty$, we have $\phi'_\lambda \to \infty$. Similarly, in the limit $\alpha \to 1^-$, i.e., $\tau \to -\infty$, we have $\phi'_\lambda \to 0$. The finiteness of $\val(\phi_\lambda)$ follows from Lemma~\ref{lem:calc-val}, where the value is also computed.
\end{proof}

\section{Analysis of the Low-Degree Likelihood}
\label{app:analysis-ldlr}

\subsection{Proof of Lemma~\ref{lem:unif-int}}
\label{app:pf-lem-unif-int}

\begin{proof}[Proof of Lemma~\ref{lem:unif-int}]
We will write the expectation as the integral of a tail bound and then apply Lemma~\ref{lem:ov-conc} (below) which gives a tail bound for $A$. Choose $\eta > 0$ small enough so that $a := \frac{1-\eta}{(1+\gamma)\lambda^2} > 1$, and let $\delta,C$ be the corresponding constants from Lemma~\ref{lem:ov-conc}. Let $\rho = \delta^2 \lambda^2/2$. Now
\begin{align*}
\EE[\One_{A \le t} \cdot \exp((1+\gamma)A)] &= \int_0^\infty \Pr\{\One_{A \le t} \cdot \exp((1+\gamma)A) \ge r\} \,dr \\
&= \int_0^\infty \Pr\{A \le t \text{ and } A \ge (1+\gamma)^{-1} \log r\} \,dr \\
&\le 1 + \int_1^{\exp((1+\gamma)t)} \Pr\{A \ge (1+\gamma)^{-1} \log r\} \,dr \\
&= 1 + \int_1^{\exp((1+\gamma)t)} \Pr\left\{\frac{|\langle x,x' \rangle|}{\|x\| \|x'\|} \ge \sqrt{\frac{2\log r}{(1+\gamma)\lambda^2 n}}\right\} \,dr \\
&\le 1 + \int_1^\infty C\exp\left(-\frac{(1-\eta)\log r}{(1+\gamma)\lambda^2}\right) \,dr
\intertext{where we have applied Lemma~\ref{lem:ov-conc}, noting that $r \le \exp((1+\gamma)t)$ implies $\sqrt{\frac{2\log r}{(1+\gamma)\lambda^2 n}} \le \sqrt{2\rho/\lambda^2} = \delta$. Continuing from above,}
&= 1 + \int_1^\infty Cr^{-a} \,dr \\
&= 1 + \frac{C}{a-1},
\end{align*}
recalling $a > 1$.
\end{proof}

We next give some concentration inequalities, building up to Lemma~\ref{lem:ov-conc} which was used above. Recall that $x,x' \in \RR^n$ are independent vectors with entries drawn i.i.d.\ from $\pi$, which satisfies Assumption~\ref{assum:pi}.

\begin{lemma}[\cite{sk-cert}, Proposition~5.12]
For every $\eta > 0$ there exist positive constants $\delta,C$ (depending on $\eta,\pi$) such that, for all $n$,
\[ \Pr\{|\langle x,x' \rangle| \ge un\} \le C \exp\left(-(1-\eta)\frac{nu^2}{2}\right) \qquad \text{for all } u \in [0,\delta]. \]
\end{lemma}

\begin{lemma}
For a constant $\zeta > 0$ (depending on $\pi$), for all $u \ge 0$ and all $n$,
\[ \Pr\left\{\left|\|x\|^2 - n\right| \ge un\right\} \le 2 \exp\left(-\frac{n}{2} \min\left\{\frac{u^2}{\zeta^2},\frac{u}{\zeta}\right\}\right) \]
(and the same holds for $x'$).
\end{lemma}
\begin{proof}
This is a form of Bernstein's inequality for subexponential random variables. Since $\pi$ is subgaussian, $\pi^2 - \EE[\pi^2]$ is subexponential. See Lemma~1.12 and Theorem~1.13 of~\cite{RH-notes}.
\end{proof}

\begin{lemma}\label{lem:ov-conc}
For every $\eta > 0$ there exist positive constants $\delta,C$ (depending on $\eta,\pi$) such that, for all $n$,
\[ \Pr\left\{\frac{|\langle x,x' \rangle|}{\|x\| \|x'\|} \ge u\right\} \le C \exp\left(-(1-\eta)\frac{nu^2}{2}\right) \qquad \text{for all } u \in [0,\delta]. \]
\end{lemma}
\begin{proof}
We combine the previous two lemmas. For $\alpha \in (0,1)$ to be chosen later,
\begin{align*}
\Pr\left\{\frac{|\langle x,x' \rangle|}{\|x\| \|x'\|} \ge u\right\} &\le \Pr\{|\langle x,x' \rangle| \ge (1-\alpha)un\} + \Pr\{\|x\|^2 \le (1-\alpha)n\} \\
&\qquad + \Pr\{\|x'\|^2 \le (1-\alpha)n\} \\
&\le C' \exp\left(-\left(1-\frac{\eta}{2}\right)\frac{n(1-\alpha)^2 u^2}{2}\right) + 4\exp\left(-\frac{n}{2} \min\left\{\frac{\alpha^2}{\zeta^2},\frac{\alpha}{\zeta}\right\}\right)
\end{align*}
for all $u \in [0,\delta'/(1-\alpha)]$, where $\delta',C',\zeta$ are positive constants depending on $\eta,\pi$. Choose $\alpha > 0$ to be a constant (depending on $\eta$) small enough so that $(1-\eta/2)(1-\alpha)^2 \ge 1-\eta$. Let $C = C'+4$ and
\[ \delta = \min\left\{\frac{\delta'}{1-\alpha} \,,\; \sqrt{\min\left\{\frac{\alpha^2}{\zeta^2},\frac{\alpha}{\zeta}\right\}}\right\}. \]
Now for all $u \in [0,\delta]$,
\begin{align*}
\Pr\left\{\frac{|\langle x,x' \rangle|}{\|x\| \|x'\|} \ge u\right\} &\le C'\exp\left(-(1-\eta)\frac{nu^2}{2}\right) + 4\exp\left(-\frac{nu^2}{2}\right) \\
&\le C\exp\left(-(1-\eta)\frac{nu^2}{2}\right),
\end{align*}
completing the proof.
\end{proof}

\subsection{Proof of Lemma~\ref{lem:L-term-3}}
\label{app:pf-lem-L-term-3}

\begin{proof}[Proof of Lemma~\ref{lem:L-term-3}]
Using Lemma~\ref{lem:ov-conc} and the scaling $t = \Omega(n)$, we have $\Pr\{A > t\} = \exp(-\Omega(n))$. From~\eqref{eq:def-A}, we have with probability 1 that $0 \le A \le n$, implying
\[ \exp^{\le D}(A) \le \exp^{\le D}(n) = \sum_{d=0}^D \frac{n^d}{d!} \le \sum_{d=0}^D n^d \le (D+1)n^D. \]
Now,
\[ \EE[\One_{A > t} \cdot \exp^{\le D}(A)] \le \Pr\{A > t\} \exp^{\le D}(n) \le \exp(-\Omega(n) + \log(D+1) + D \log n), \]
which is $o(1)$ because $D = o(n/\log n)$.
\end{proof}

\subsection{Proof of Theorem~\ref{thm:L-bound}}
\label{app:pf-thm-L-bound}

\begin{proof}[Proof of Theorem~\ref{thm:L-bound}]
Recall the assumptions $D = \omega(1)$ and $D = o(n/\log n)$. Recall~\eqref{eq:L-decomp} and~\eqref{eq:L-decomp-2}, whose terms are bounded in Lemmas~\ref{lem:L-term-1}, \ref{lem:L-term-2}, and~\ref{lem:L-term-3}.

To prove the lower bound $\|(L_\lambda)^{\le D}\|^2 \ge (1-\lambda^2)^{-1/2} - o(1)$, choose $t = t_n$ such that $t = \omega(1)$ and $t = o(D)$ (for instance, $t = \sqrt{D}$). Since the final term in~\eqref{eq:L-decomp} is non-negative,
\[ \|(L_\lambda)^{\le D}\|^2 \ge \EE[\One_{A \le t} \cdot \exp(A)] - \sum_{d=D+1}^\infty \frac{t^d}{d!} = (1-\lambda^2)^{-1/2} - o(1), \]
using Lemmas~\ref{lem:L-term-1} and~\ref{lem:L-term-2}.

To prove the upper bound $\|(L_\lambda)^{\le D}\|^2 \le (1-\lambda^2)^{-1/2} + o(1)$, choose $t = \rho n$ where $\rho > 0$ is the constant from Lemma~\ref{lem:L-term-2}. We have
\[ \|(L_\lambda)^{\le D}\|^2 \le \EE[\One_{A \le t} \cdot \exp(A)] + \EE[\One_{A > t} \cdot \exp^{\le D}(A)] = (1-\lambda^2)^{-1/2} + o(1), \]
using Lemmas~\ref{lem:L-term-2} and~\ref{lem:L-term-3}.
\end{proof}

\section{Lemmas for Proposition~\ref{prop:reduction}}

\subsection{Proof of Lemma~\ref{lem:construct-v}}
\label{app:pf-lem-construct-v}

\begin{proof}[Proof of Lemma~\ref{lem:construct-v}]
Write $\slope(\cdot,\cdot)$ for the slope of the line connecting two points, and use the shorthand $m_i = \slope(u_i,u_{i+1})$ and $\ell_i = \slope(v_i,v_{i+1})$, where, recall, $u_i = (a_i,\psi(a_i))$. Let $\gamma > 0$ be a constant to be chosen later. For $i = 1,\ldots,r-1$, choose $b_i < \psi(a_i)$ such that $\slope(u_{i-1},v_i) \ge (1-\gamma)m_{i-1}$ and $\slope(u_{i-1},v_i) > \slope(v_i,u_{i+1})$, which is possible because $u$ is in concave position. Note that
\begin{equation}\label{eq:ell-m-ratio}
\ell_i \ge \slope(u_i,v_{i+1}) \ge (1-\gamma)m_i \qquad\text{for } i = 0,\ldots,r-2
\end{equation}
and $\ell_{r-1} \ge m_{r-1}$. This implies $\val(\conc(v)) \ge (1-\gamma) \val(\conc(u))$, so using Lemma~\ref{lem:construct-u} we have $\val(\conc(v)) > \val(\phi)$ for a sufficiently small choice of $\gamma > 0$. To verify that $v$ is in concave position,
\[ \ell_{i-1} \ge \slope(u_{i-1},v_i) > \slope(v_i,u_{i+1}) \ge \ell_i, \]
and from \eqref{eq:ell-m-ratio}, $\ell_i \ge (1-\gamma)m_i > 0$.
\end{proof}

\subsection{Proof of Lemma~\ref{lem:alt-sum}}
\label{app:pf-lem-alt-sum}

\begin{proof}[Proof of Lemma~\ref{lem:alt-sum}]
Assume $m$ is even without loss of generality, as otherwise we can increase $m$ by one and add an extra term $h_{m-1} = 0$. The first inequality holds because each negative term can be paired with the preceding positive term, whose magnitude is strictly larger. Similarly, the argument to the square root is nonnegative. It therefore suffices to show
\[ \left(\sum_{i=0}^{m-1} (-1)^i h_i\right)^2 \le \sum_{i=0}^{m-1} (-1)^i h_i^2. \]
This identity admits a slick geometric proof by interpreting each side as the area of some region in $\RR^2$: the left-hand side is the area of the rectangles
\[ R_1 := ([h_{m-1},h_{m-2}] \cup [h_{m-3},h_{m-4}] \cup \cdots \cup [h_1,h_0])^{\times 2} \]
and the right-hand side is the area of the L-shaped regions
\[ R_2 := ([0,h_0]^{\times 2} \setminus [0,h_1]^{\times 2}) \cup ([0,h_2]^{\times 2} \setminus [0,h_3]^{\times 2}) \cup \cdots \cup ([0,h_{m-2}]^{\times 2} \setminus [0,h_{m-1}]^{\times 2}). \]
Since $R_1 \subseteq R_2$, the proof is complete.
\end{proof}

\section{Update on the Main Conjecture}
\label{app:update}

The original version of Conjecture~\ref{conj:new} that we posed in November 2023 read as follows.

\begin{conjecture}
Consider the spiked Wigner testing problem (Definition~\ref{def:wigner}). Fix any $0 < \delta_1 < \delta_2$, any $\lambda \in (0,1)$, and any $\pi$ satisfying Assumption~\ref{assum:pi}. Any sequence of functions $f = f_n$ computable in time $\exp(O(n^{\delta_1}))$ must satisfy
\begin{equation}\label{eq:conj-ratio-app}
\limsup_{n \to \infty} R_\lambda(f) \le \limsup_{n \to \infty} \sup_{g \,:\, \deg(g) \le n^{\delta_2}} R_\lambda(g).
\end{equation}
\end{conjecture}

This version of the conjecture was refuted in March 2025 by Ansh Nagda (personal communication). The issue is that a poly-time algorithm $f$ can ``cheat'' by exploiting rare events. Specifically, if $\pi$ is the sparse Rademacher prior (see Remark~\ref{rem:large-D}) with $\rho > 0$ a sufficiently small constant, $f$ can bet on the event that the signal is equal to one specific vector $v$, and output 1 when detecting this event (which is done by thresholding $v^\top Y v$), and 0 otherwise. This makes the left-hand side of~\eqref{eq:conj-ratio-app} unbounded for some $\lambda < 1$. In Conjecture~\ref{conj:new} we have implemented a plausible fix by requiring $f$ to output values in the range $[1/B,B]$ for a constant $B \ge 1$. Our main result, Corollary~\ref{cor:main-comp}, still holds under this modified conjecture without any changes to the proof, since Proposition~\ref{prop:reduction} produces a function $f$ whose outputs come from a finite list of positive constants.

\addcontentsline{toc}{section}{Acknowledgments}

\section*{Acknowledgments}

We thank Ahmed El Alaoui for helpful discussions.

\addcontentsline{toc}{section}{References}

\bibliographystyle{alpha}
\bibliography{main}

\newcommand{\etalchar}[1]{$^{#1}$}
\begin{thebibliography}{PWBM18}

\bibitem[ALR87]{ALR}
Michael Aizenman, Joel~L Lebowitz, and David Ruelle.
\newblock Some rigorous results on the {Sherrington-Kirkpatrick} spin glass
  model.
\newblock {\em Communications in mathematical physics}, 112:3--20, 1987.

\bibitem[AZ06]{clt-band}
Greg~W Anderson and Ofer Zeitouni.
\newblock A {CLT} for a band matrix model.
\newblock {\em Probability Theory and Related Fields}, 134(2):283--338, 2006.

\bibitem[BB20]{secret-leakage}
Matthew Brennan and Guy Bresler.
\newblock Reducibility and statistical-computational gaps from secret leakage.
\newblock In {\em Conference on Learning Theory}, pages 648--847. PMLR, 2020.

\bibitem[BBK{\etalchar{+}}21]{spectral-planting}
Afonso~S Bandeira, Jess Banks, Dmitriy Kunisky, Christopher Moore, and
  Alexander~S Wein.
\newblock Spectral planting and the hardness of refuting cuts, colorability,
  and communities in random graphs.
\newblock In {\em Conference on Learning Theory}, pages 410--473. PMLR, 2021.

\bibitem[BBP05]{BBP}
Jinho Baik, G{\'e}rard {Ben Arous}, and Sandrine P{\'e}ch{\'e}.
\newblock Phase transition of the largest eigenvalue for nonnull complex sample
  covariance matrices.
\newblock {\em Annals of Probability}, pages 1643--1697, 2005.

\bibitem[BEH{\etalchar{+}}22]{fp}
Afonso~S Bandeira, Ahmed {El Alaoui}, Samuel Hopkins, Tselil Schramm,
  Alexander~S Wein, and Ilias Zadik.
\newblock The {Franz-Parisi} criterion and computational trade-offs in high
  dimensional statistics.
\newblock {\em Advances in Neural Information Processing Systems},
  35:33831--33844, 2022.

\bibitem[BHK{\etalchar{+}}19]{sos-clique}
Boaz Barak, Samuel Hopkins, Jonathan Kelner, Pravesh~K Kothari, Ankur Moitra,
  and Aaron Potechin.
\newblock A nearly tight sum-of-squares lower bound for the planted clique
  problem.
\newblock {\em SIAM Journal on Computing}, 48(2):687--735, 2019.

\bibitem[BKW20]{sk-cert}
Afonso~S Bandeira, Dmitriy Kunisky, and Alexander~S Wein.
\newblock Computational hardness of certifying bounds on constrained {PCA}
  problems.
\newblock In {\em 11th Innovations in Theoretical Computer Science Conference
  (ITCS 2020)}. Schloss Dagstuhl-Leibniz-Zentrum f{\"u}r Informatik, 2020.

\bibitem[BL16]{BL-free-1}
Jinho Baik and Ji~Oon Lee.
\newblock Fluctuations of the free energy of the spherical
  {Sherrington--Kirkpatrick} model.
\newblock {\em Journal of Statistical Physics}, 165:185--224, 2016.

\bibitem[BL17]{BL-free-2}
Jinho Baik and Ji~Oon Lee.
\newblock Fluctuations of the free energy of the spherical
  {Sherrington--Kirkpatrick} model with ferromagnetic interaction.
\newblock In {\em Annales Henri Poincar{\'e}}, volume~18, pages 1867--1917.
  Springer, 2017.

\bibitem[BM11]{BM-amp}
Mohsen Bayati and Andrea Montanari.
\newblock The dynamics of message passing on dense graphs, with applications to
  compressed sensing.
\newblock {\em IEEE Transactions on Information Theory}, 57(2):764--785, 2011.

\bibitem[BM17]{banerjee-ma}
Debapratim Banerjee and Zongming Ma.
\newblock Optimal hypothesis testing for stochastic block models with growing
  degrees.
\newblock {\em arXiv preprint arXiv:1705.05305}, 2017.

\bibitem[BMV{\etalchar{+}}18]{BMVVX}
Jess Banks, Cristopher Moore, Roman Vershynin, Nicolas Verzelen, and Jiaming
  Xu.
\newblock Information-theoretic bounds and phase transitions in clustering,
  sparse {PCA}, and submatrix localization.
\newblock {\em IEEE Transactions on Information Theory}, 64(7):4872--4894,
  2018.

\bibitem[BN11]{BN-eigenvec}
Florent {Benaych-Georges} and Raj~Rao Nadakuditi.
\newblock The eigenvalues and eigenvectors of finite, low rank perturbations of
  large random matrices.
\newblock {\em Advances in Mathematics}, 227(1):494--521, 2011.

\bibitem[BR13]{BR-reduction}
Quentin Berthet and Philippe Rigollet.
\newblock Complexity theoretic lower bounds for sparse principal component
  detection.
\newblock In {\em Conference on learning theory}, pages 1046--1066. PMLR, 2013.

\bibitem[BS04]{clt-wishart}
ZD~Bai and Jack~W Silverstein.
\newblock {CLT} for linear spectral statistics of large-dimensional sample
  covariance matrices.
\newblock {\em The Annals of Probability}, 32(1A):553--605, 2004.

\bibitem[BS06]{BS-spiked}
Jinho Baik and Jack~W Silverstein.
\newblock Eigenvalues of large sample covariance matrices of spiked population
  models.
\newblock {\em Journal of multivariate analysis}, 97(6):1382--1408, 2006.

\bibitem[BSS89]{BSS}
Lenore Blum, Mike Shub, and Steve Smale.
\newblock On a theory of computation and complexity over the real numbers:
  Np-completeness, recursive functions and universal machines.
\newblock {\em Bulletin of the American Mathematical Society}, 21(1):1--46,
  1989.

\bibitem[BWZ09]{clt-wigner}
Zhidong Bai, Xiaoying Wang, and Wang Zhou.
\newblock {CLT} for linear spectral statistics of {Wigner} matrices.
\newblock {\em Electronic Journal of Probability}, 14:2391--2417, 2009.

\bibitem[BY08]{BY-limit}
Zhidong Bai and Jian-feng Yao.
\newblock Central limit theorems for eigenvalues in a spiked population model.
\newblock In {\em Annales de l'IHP Probabilit{\'e}s et statistiques},
  volume~44, pages 447--474, 2008.

\bibitem[CDF09]{CDF-wigner}
Mireille Capitaine, Catherine {Donati-Martin}, and Delphine F{\'e}ral.
\newblock The largest eigenvalues of finite rank deformation of large {Wigner}
  matrices: convergence and nonuniversality of the fluctuations.
\newblock {\em The Annals of Probability}, 37(1):1--47, 2009.

\bibitem[CL22]{weak-wigner}
Hye~Won Chung and Ji~Oon Lee.
\newblock Weak detection in the spiked {Wigner} model.
\newblock {\em IEEE Transactions on Information Theory}, 68(11):7427--7453,
  2022.

\bibitem[DK22]{lattice-2}
Ilias Diakonikolas and Daniel Kane.
\newblock Non-gaussian component analysis via lattice basis reduction.
\newblock In {\em Conference on Learning Theory}, pages 4535--4547. PMLR, 2022.

\bibitem[DKMZ11]{decelle}
Aurelien Decelle, Florent Krzakala, Cristopher Moore, and Lenka Zdeborov{\'a}.
\newblock Asymptotic analysis of the stochastic block model for modular
  networks and its algorithmic applications.
\newblock {\em Physical Review E}, 84(6):066106, 2011.

\bibitem[DKWB23]{subexp-sparse}
Yunzi Ding, Dmitriy Kunisky, Alexander~S Wein, and Afonso~S Bandeira.
\newblock Subexponential-time algorithms for sparse {PCA}.
\newblock {\em Foundations of Computational Mathematics}, pages 1--50, 2023.

\bibitem[DMM09]{amp}
David~L Donoho, Arian Maleki, and Andrea Montanari.
\newblock Message-passing algorithms for compressed sensing.
\newblock {\em Proceedings of the National Academy of Sciences},
  106(45):18914--18919, 2009.

\bibitem[DW23]{global-local}
Xiucai Ding and Zhenggang Wang.
\newblock Global and local {CLTs} for linear spectral statistics of general
  sample covariance matrices when the dimension is much larger than the sample
  size with applications.
\newblock {\em arXiv preprint arXiv:2308.08646}, 2023.

\bibitem[EKJ20]{fund-limits-wigner}
Ahmed {El Alaoui}, Florent Krzakala, and Michael~I Jordan.
\newblock Fundamental limits of detection in the spiked {Wigner} model.
\newblock {\em The Annals of Statistics}, 48(2):863--885, 2020.

\bibitem[FGR{\etalchar{+}}17]{sq-clique}
Vitaly Feldman, Elena Grigorescu, Lev Reyzin, Santosh~S Vempala, and Ying Xiao.
\newblock Statistical algorithms and a lower bound for detecting planted
  cliques.
\newblock {\em Journal of the ACM}, 64(2):1--37, 2017.

\bibitem[FP07]{FP-wigner}
Delphine F{\'e}ral and Sandrine P{\'e}ch{\'e}.
\newblock The largest eigenvalue of rank one deformation of large {Wigner}
  matrices.
\newblock {\em Communications in mathematical physics}, 272:185--228, 2007.

\bibitem[FR18]{FR-amp}
Alyson~K Fletcher and Sundeep Rangan.
\newblock Iterative reconstruction of rank-one matrices in noise.
\newblock {\em Information and Inference: A Journal of the IMA}, 7(3):531--562,
  2018.

\bibitem[GJW20]{GJW-ld}
David Gamarnik, Aukosh Jagannath, and Alexander~S Wein.
\newblock Low-degree hardness of random optimization problems.
\newblock In {\em 61st Annual Symposium on Foundations of Computer Science
  (FOCS)}, pages 131--140. IEEE, 2020.

\bibitem[GMZ22]{phys-survey}
David Gamarnik, Cristopher Moore, and Lenka Zdeborov{\'a}.
\newblock Disordered systems insights on computational hardness.
\newblock {\em Journal of Statistical Mechanics: Theory and Experiment},
  2022(11):114015, 2022.

\bibitem[GS17]{GS-ogp}
David Gamarnik and Madhu Sudan.
\newblock Limits of local algorithms over sparse random graphs.
\newblock {\em The Annals of Probability}, pages 2353--2376, 2017.

\bibitem[HKP{\etalchar{+}}17]{sos-detect}
Samuel~B Hopkins, Pravesh~K Kothari, Aaron Potechin, Prasad Raghavendra, Tselil
  Schramm, and David Steurer.
\newblock The power of sum-of-squares for detecting hidden structures.
\newblock In {\em 58th Annual Symposium on Foundations of Computer Science
  (FOCS)}, pages 720--731. IEEE, 2017.

\bibitem[Hop18]{hopkins-thesis}
Samuel Hopkins.
\newblock {\em Statistical Inference and the Sum of Squares Method}.
\newblock PhD thesis, Cornell University, 2018.

\bibitem[HS17]{HS-bayesian}
Samuel~B Hopkins and David Steurer.
\newblock Efficient bayesian estimation from few samples: community detection
  and related problems.
\newblock In {\em 58th Annual Symposium on Foundations of Computer Science
  (FOCS)}, pages 379--390. IEEE, 2017.

\bibitem[HW21]{HW-counter}
Justin Holmgren and Alexander~S Wein.
\newblock Counterexamples to the low-degree conjecture.
\newblock In {\em 12th Innovations in Theoretical Computer Science Conference
  (ITCS 2021)}, volume 185, 2021.

\bibitem[JCL20]{weak-wigner-general-rank}
Ji~Hyung Jung, Hye~Won Chung, and Ji~Oon Lee.
\newblock Weak detection in the spiked {Wigner} model with general rank.
\newblock {\em arXiv preprint arXiv:2001.05676}, 2020.

\bibitem[JCL21]{det-rect}
Ji~Hyung Jung, Hye~Won Chung, and Ji~Oon Lee.
\newblock Detection of signal in the spiked rectangular models.
\newblock In {\em International Conference on Machine Learning}, pages
  5158--5167. PMLR, 2021.

\bibitem[Jer92]{jerrum}
Mark Jerrum.
\newblock Large cliques elude the metropolis process.
\newblock {\em Random Structures \& Algorithms}, 3(4):347--359, 1992.

\bibitem[Joh01]{johnstone}
Iain~M Johnstone.
\newblock On the distribution of the largest eigenvalue in principal components
  analysis.
\newblock {\em The Annals of statistics}, 29(2):295--327, 2001.

\bibitem[KM21]{KM-tree}
Frederic Koehler and Elchanan Mossel.
\newblock Reconstruction on trees and low-degree polynomials.
\newblock {\em arXiv preprint arXiv:2109.06915}, 2021.

\bibitem[Kun21]{morris}
Dmitriy Kunisky.
\newblock Hypothesis testing with low-degree polynomials in the {Morris} class
  of exponential families.
\newblock In {\em Conference on Learning Theory}, pages 2822--2848. PMLR, 2021.

\bibitem[KVWX23]{coloring-clique}
Pravesh Kothari, Santosh~S Vempala, Alexander~S Wein, and Jeff Xu.
\newblock Is planted coloring easier than planted clique?
\newblock In {\em The Thirty Sixth Annual Conference on Learning Theory}, pages
  5343--5372. PMLR, 2023.

\bibitem[KWB19]{ld-notes}
Dmitriy Kunisky, Alexander~S Wein, and Afonso~S Bandeira.
\newblock Notes on computational hardness of hypothesis testing: Predictions
  using the low-degree likelihood ratio.
\newblock In {\em ISAAC Congress (International Society for Analysis, its
  Applications and Computation)}, pages 1--50. Springer, 2019.

\bibitem[LHBS22]{lss-div-spikes}
Zhijun Liu, Jiang Hu, Zhidong Bai, and Haiyan Song.
\newblock A {CLT} for the {LSS} of large dimensional sample covariance matrices
  with diverging spikes.
\newblock {\em arXiv preprint arXiv:2212.05896}, 2022.

\bibitem[LKZ15]{LKZ-sparse}
Thibault Lesieur, Florent Krzakala, and Lenka Zdeborov{\'a}.
\newblock Phase transitions in sparse {PCA}.
\newblock In {\em International Symposium on Information Theory (ISIT)}, pages
  1635--1639. IEEE, 2015.

\bibitem[LR05]{testing-book}
EL~Lehmann and Joseph~P Romano.
\newblock Testing statistical hypotheses.
\newblock {\em Springer texts in statistics}, 2005.

\bibitem[LWB22]{sparse-clustering}
Matthias L{\"o}ffler, Alexander~S Wein, and Afonso~S Bandeira.
\newblock Computationally efficient sparse clustering.
\newblock {\em Information and Inference: A Journal of the IMA},
  11(4):1255--1286, 2022.

\bibitem[Ma{\"\i}07]{maida}
Myl{\`e}ne Ma{\"\i}da.
\newblock Large deviations for the largest eigenvalue of rank one deformations
  of gaussian ensembles.
\newblock {\em Electronic Journal of Probability}, 12:1131--1150, 2007.

\bibitem[Mio18]{miolane-survey}
L{\'e}o Miolane.
\newblock Phase transitions in spiked matrix estimation: information-theoretic
  analysis.
\newblock {\em arXiv preprint arXiv:1806.04343}, 2018.

\bibitem[MRZ15]{MRZ}
Andrea Montanari, Daniel Reichman, and Ofer Zeitouni.
\newblock On the limitation of spectral methods: From the gaussian hidden
  clique problem to rank-one perturbations of gaussian tensors.
\newblock {\em Advances in Neural Information Processing Systems}, 28, 2015.

\bibitem[MV21]{MV-amp}
Andrea Montanari and Ramji Venkataramanan.
\newblock Estimation of low-rank matrices via approximate message passing.
\newblock {\em The Annals of Statistics}, 49(1):321--345, 2021.

\bibitem[MW22]{MW-amp}
Andrea Montanari and Alexander~S Wein.
\newblock Equivalence of approximate message passing and low-degree polynomials
  in rank-one matrix estimation.
\newblock {\em arXiv preprint arXiv:2212.06996}, 2022.

\bibitem[Nad08]{nadler}
Boaz Nadler.
\newblock Finite sample approximation results for principal component analysis:
  a matrix perturbation approach.
\newblock {\em The Annals of Statistics}, 36(6):2791--2817, 2008.

\bibitem[OMH13]{sphericity}
Alexei Onatski, Marcelo~J Moreira, and Marc Hallin.
\newblock Asymptotic power of sphericity tests for high-dimensional data.
\newblock {\em The Annals of Statistics}, 41(3):1204--1231, 2013.

\bibitem[OMH14]{sphericity-2}
Alexei Onatski, Marcelo~J Moreira, and Marc Hallin.
\newblock Signal detection in high dimension: The multispiked case.
\newblock {\em The Annals of Statistics}, pages 225--254, 2014.

\bibitem[Pau07]{paul}
Debashis Paul.
\newblock Asymptotics of sample eigenstructure for a large dimensional spiked
  covariance model.
\newblock {\em Statistica Sinica}, pages 1617--1642, 2007.

\bibitem[P{\'e}c06]{peche}
Sandrine P{\'e}ch{\'e}.
\newblock The largest eigenvalue of small rank perturbations of {Hermitian}
  random matrices.
\newblock {\em Probability Theory and Related Fields}, 134:127--173, 2006.

\bibitem[PRS13]{PRS-wigner}
Alessandro Pizzo, David Renfrew, and Alexander Soshnikov.
\newblock On finite rank deformations of {Wigner} matrices.
\newblock In {\em Annales de l'IHP Probabilit{\'e}s et Statistiques},
  volume~49, pages 64--94, 2013.

\bibitem[PWBM18]{opt-subopt}
Amelia Perry, Alexander~S Wein, Afonso~S Bandeira, and Ankur Moitra.
\newblock Optimality and sub-optimality of {PCA I}: Spiked random matrix
  models.
\newblock {\em The Annals of Statistics}, 46(5):2416--2451, 2018.

\bibitem[RH23]{RH-notes}
Philippe Rigollet and Jan-Christian H{\"u}tter.
\newblock High-dimensional statistics (lecture notes).
\newblock {\em arXiv preprint arXiv:2310.19244}, 2023.

\bibitem[RSS18]{sos-survey}
Prasad Raghavendra, Tselil Schramm, and David Steurer.
\newblock High dimensional estimation via sum-of-squares proofs.
\newblock In {\em Proceedings of the International Congress of Mathematicians:
  Rio de Janeiro 2018}, pages 3389--3423. World Scientific, 2018.

\bibitem[RSWY23]{planted-v-planted}
Cynthia Rush, Fiona Skerman, Alexander~S Wein, and Dana Yang.
\newblock Is it easier to count communities than find them?
\newblock In {\em 14th Innovations in Theoretical Computer Science Conference
  (ITCS 2023)}. Schloss Dagstuhl-Leibniz-Zentrum f{\"u}r Informatik, 2023.

\bibitem[SW22]{SW-estimation}
Tselil Schramm and Alexander~S Wein.
\newblock Computational barriers to estimation from low-degree polynomials.
\newblock {\em The Annals of Statistics}, 50(3):1833--1858, 2022.

\bibitem[WY21]{lss-block}
Zhenggang Wang and Jianfeng Yao.
\newblock Central limit theorem for linear spectral statistics of
  block-{Wigner}-type matrices.
\newblock {\em arXiv preprint arXiv:2110.12171}, 2021.

\bibitem[ZSWB22]{lattice-1}
Ilias Zadik, Min~Jae Song, Alexander~S Wein, and Joan Bruna.
\newblock Lattice-based methods surpass sum-of-squares in clustering.
\newblock In {\em Conference on Learning Theory}, pages 1247--1248. PMLR, 2022.

\end{thebibliography}

\end{document}